\documentclass[12pt]{paper}
 \usepackage{geometry}
\usepackage{hyperref}
                \usepackage{amsthm,amsfonts,amsmath,amscd,amssymb,epsfig,verbatim,
}

\usepackage{graphicx}
\usepackage{epstopdf}
 {\title{Lagrangian caps} 
\author{ Yakov Eliashberg\thanks{Partially supported by the NSF grant DMS-1205349} \\ Stanford University  \and   Emmy Murphy
\\ MIT} 
\date{}  
 \setlength{\marginparwidth}{1.2in}
\let\oldmarginpar\marginpar
\renewcommand\marginpar[1]{\-\oldmarginpar[\raggedleft\footnotesize #1]%
{\raggedright\footnotesize #1}}

\parindent=0pt
\parskip=4pt
\hyphenation{ma-ni-fold ma-ni-folds sub-ma-ni-fold sub-ma-ni-folds}
\theoremstyle{plain}
\newtheorem{theorem}{Theorem}[section]
\newtheorem{thm}[theorem]{Theorem}

\newtheorem{cor}[theorem]{Corollary}

\newtheorem{prop}[theorem]{Proposition}
\newtheorem{lemma}[theorem]{Lemma}

\theoremstyle{remark}

\newtheorem{remark}[theorem]{Remark}
\newtheorem*{remark*}{Remark}
\newtheorem*{example*}{Example}

\theoremstyle{definition}

\newcommand{\wt}{\widetilde}
\newcommand{\wh}{\widehat}

\newcommand{\og}{\overline{\gamma}}

\newcommand{\oeta}{\overline{\eta}}

\newcommand{\p}{\partial}

\newcommand{\om}{\omega}

\newcommand{\eps}{\varepsilon}

\newcommand{\Z}{{\mathbb{Z}}}
\newcommand{\R}{{\mathbb{R}}}
\newcommand{\C}{{\mathbb{C}}}

\newcommand{\st}{{\rm st}}

\newcommand{\Int}{{\rm Int\,}} %Interior

\renewcommand{\min}{{\rm min}}
\renewcommand{\max}{{\rm max}}

\newcommand{\std}{{\rm std}}

\newcommand{\tb}{{\rm tb}}

\newcommand{\Id}{\mathrm {Id}}

\newcommand{\SI}{\mathrm{SI}}

\def\Op{{\mathcal O}{\it p}\,}
\numberwithin{figure}{section}

\begin{document}
\maketitle
\begin{abstract}
We establish an $h$-principle for exact Lagrangian embeddings with concave Legendrian boundary. We prove, in particular, that in the complement of the unit ball $B$ in the standard symplectic $\R^{2n}, 2n\geq 6$, there exists an embedded Lagrangian $n$-disc    transversely attached  to $B$  along its  Legendrian boundary.
\end{abstract}

%\tableofcontents
\section{Introduction}

{\bf Question}.
Let $B$ be the round ball in the standard symplectic $\R^{2n}$. {\it Is there an embedded  Lagrangian disc $\Delta\subset \R^{2n}\setminus \Int B$ with
$\p \Delta\subset\p B$ such that $\p \Delta$ is a Legendrian submanifold and $\Delta$ transversely intersects $\p B$ along its boundary?}
\medskip

If $n=2$ then such a  Lagrangian  disc  does not exist. Indeed, it is easy to check that the existence of such a Lagrangian disc implies that  the Thurston-Bennequin invariant $\tb(\p\Delta)$of the Legendrian knot $\p\Delta\subset S^3$  is equal to $+1$. On the other hand, the knot $\p\Delta$ is sliced, i.e its $4$-dimensional genus is equal to $0$.
But then according to Lee Rudolph's   slice Bennequin inequality \cite{Rudolph} we should have $\tb(\p\Delta)\leq -1$, which is a contradiction.
 \medskip

 As far as we know no such  Lagrangian discs have been previously constructed in higher dimensions either. We prove in this paper that if $n>2$ such discs exist in abundance.  In particular, we prove
\begin{theorem} \label{thm:caps} Let $L$ be a smooth manifold  of dimension $n>2$ with non-empty  boundary such that  its complexified tangent bundle $T(L)\otimes\C$ is trivial.
Then   there exists  
   an exact Lagrangian embedding  $f:(L,\p L)\to( \R^{2n}\setminus \Int B,\p B)$ with
$f(\p \Delta)\subset\p B$ such that $f(\p \Delta)\subset\p B$ is a Legendrian submanifold and  $f$  transverse  to  $\p B$ along the boundary $\p L$.
\end{theorem} 
Note that the triviality of the bundle $T(L)\otimes\C$ is a necessary (and according to Gromov's $h$-principle for Lagrangian immersions, \cite{Gr-PDR} sufficient) condition for existence of any Lagrangian {\it immersion }$L\to\C^n$.

In fact, we prove a very general  $h$-principle type result for   Lagrangian embeddings generalizing this claim,   see Theorem \ref{thm:main} below.
As corollaries of this theorem we get
\begin{itemize}
\item an  $h$-principle for Lagrangian embeddings in any symplectic manifold with a unique conical singular point, see Corollary \ref{cor:conic};
  \item  a general $h$-principle for embeddings of flexible Weinstein domains, see Corollary \ref{cor:flexible-embed};
  \item construction of Lagrangian immersions with minimal number of self-intersection points;  this is explored  in a   joint paper  of the authors with T.~Ekholm and I.~Smith, \cite{EkElMuSm}.  
 
  \end{itemize}
   
 Theorem  \ref{thm:main} together with the results from the book
 \cite{CieEli-Stein} yield new examples of rationally convex domains in $\C^n$, which will be discussed elsewhere. The authors are thankful to   Stefan Nemirovski,  whose questions  concerning this circle of questions  motivated   the results of the current paper.

\section{Main Theorem}\label{sec:main-theorem}
\subsubsection*{Loose Legendrian submanifolds}
Let $(Y,\xi)$  be a $(2n-1)$-dimensional contact manifold.
Let us recall that each contact plane $\xi_y$, $y\in Y$, carries a canonical linear symplectic structure defined up to a scaling factor. Thus, there is a well defined class of isotropic and, in particular, Lagrangian  linear subspaces of $\xi_y$.
 Given a $k$-dimensional , $k\leq n-1$, manifold $\Lambda$, an injective  homomorphism $\Phi:T\Lambda\to TY$ covering a map $\phi:\Lambda\to Y$ is called isotropic (or if $k=n-1$ Legendrian) if $\Phi(T\Lambda)\subset\xi$ and  $\Phi(T_x\Lambda)\subset \xi_{\phi(x)}$ is isotropic for each $x\in \Lambda$.
 
Given a $(2n-1)$-dimensional contact manifold $(Y,\xi)$, an  embedding $f:\Lambda\to Y$ is called {\it isotropic} if it is tangent to $\xi$; if  in addition $\dim \Lambda=n-1$ then it is called {\it Legendrian}.
 The differential of an isotropic (resp. Legendrian) embedding is an isotropic (resp. Legendrian) homomorphism. 
 
 Two Legendrian embeddings   $f_0,f_1:\Lambda\to Y$ are called {\it  formally Legendrian isotopic}
if there exists a smooth isotopy $f_t:\Lambda\to Y$ connecting $f_0$ and $ f_1$ and a $2$-parametric family of injective homomorphisms $\Phi_t^s:T\Lambda\to TY$, such that $\Phi_t^0=df_t, \Phi_0^s=df_0, \Phi_1^s=df_1$ and $\Phi^1_t$ is a  Legendrian homomorphism ($s,t\in[0,1]$).

The results of this paper essentially depend on the theory of {\it loose Legendrian  } embeddings developed in \cite{Murphy-loose}. This is a class of Legendrian embeddings into contact manifolds of dimension $>3$ which satisfy a certain form of an $h$-principle. 
For the purposes of this paper we  will not need a formal definition of   loose Legendrian embeddings, but instead just describe their properties.  

 Let  $\R^{2n-1}_{\std}:=(\R^{2n-1},\xi_\std=\{dz-\sum\limits_1^{n-1}y_idx_i=0\})$ be the standard contact $\R^{2n-1}$, $n>2$, and  $\Lambda_0 \subset \R^{2n-1}_{\std}$ be the Legendrian $\{z=0, y_i=0\}$.  Note that a small neighborhood of any point on a Legendrian  in a contact manifold is  contactomorphic to the pair $(\R^{2n-1}_\std,\Lambda_0)$.   There is another Legendrian $\tilde{\Lambda}$, called the \emph{universal loose Legendrian}, which is equal to $\Lambda_0$ outside of a compact subset, and formally Legendrian isotopic to it. A picture of $\tilde{\Lambda}$ is given in Figure ~\ref{fig:stab}, though we do not use any properties of $\Lambda$ besides those stated above.%%
       \begin{figure}
 \center{\includegraphics[scale = .8]{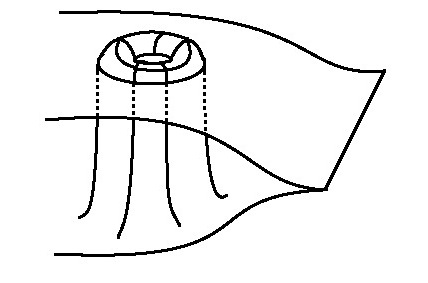}} \caption{The universal loose Legendrian, $\tilde{\Lambda}$. In the terminology of  \cite{Murphy-loose}  and \cite{CieEli-Stein}  $\tilde{\Lambda}$ is 
  the stabilization of $\Lambda_0$ over a  manifold of Euler characteristic $0$.} \label{fig:stab}
\end{figure}
  A {\it connected} Legendrian submanifold $\Lambda\subset Y$ is called 
  \emph{loose},  if there is a contact embedding $(\R^{2n-1}_{\std},\wt\Lambda)\to (Y,\Lambda)$.
 We refer the interested readers to the paper \cite{Murphy-loose} and the book \cite{CieEli-Stein} for  more information. 
 The following proposition summarizes the properties of loose  Legendrian embeddings.
 
\begin{prop}\label{prop:Murphy}
For any contact manifold $(Y,\xi)$ of dimension $2n-1>3$ the set of connected loose Legendrians have the following properties:
\begin{enumerate}
\item  For any  Legendrian embedding $f:\Lambda \to Y$ there is a loose Legendrian embedding  $\wt f:\Lambda \to Y$ which coincides with $f$ outside an arbitrarily small neighborhood of a  point $p\in\Lambda$ and which is formally isotopic to $f$ via a formal Legendrian isotopy supported in this neighborhood. 

\item  Let  $f_0,f_1:\Lambda\to Y$ be two loose Legendrian embeddings of a connected $\Lambda$ which coincide outside a compact set and which are formally Legendrian isotopic via a compactly supported isotopy. Then $f_0,f_1$ are Legendrian isotopic via a compactly supported Legendrian isotopy.

\item Let $f_t:\Lambda\to Y$, $t\in[0,1]$, be a smooth isotopy which begins with a lose Legendrian embedding $f_0$. Then it can be $C^0$-approximated by a Legendrian isotopy $\wt f_t:\Lambda\to Y$, $t\in[0,1],$ beginning with $\wt f_0=f_0$.
 \end{enumerate}
\end{prop}

 Statement (i) is the   {\it Legendrian stabilization} construction which replaces a small neighborhood of a point on a Legendrian submanifold  by the model 
$(\R^{2n-1}_{\std},\wt\Lambda)$. It  was first     described for $n>2$ in \cite{Eli-Stein}. The main part of Proposition \ref{prop:Murphy}, parts (ii) and (iii),
  are proven in \cite{Murphy-loose}. Notice that (ii)   implies   that  if a Legendrian is already loose that any further stabilizations do not change its Legendrian isotopy class. 

\subsubsection*{Symplectic manifolds with negative Liouville ends}
  
 Throughout the paper we use the terms {\it closed  submanifold}  and {\it properly embedded submanifold} as synonyms, meaning   a submanifold which is a closed subset, but not necessarily a closed manifold itself. 
 
   Let $L$ be an $n$-dimensional smooth manifold. 
 A {\it negative end} structure on $L$ is a choice of 
\begin{itemize}
\item a codimension $1$   submanifold  $\Lambda \subset L$  which divides $L$ into two parts: $L=L_-\cup L_+$, $L_-\cap L_+=\Lambda$, and
\item   a non-vanishing vector field $S$ on $\Op L_-\subset L$ which is outward  transverse to the boundary $\Lambda=\p L_-$, and such that the negative flow $S^{-t}:L_-\to L_-$ is defined for all $t$ and all its  trajectories intersect $\Lambda$.
\end{itemize}
In other words, there is a canonical diffeomorphism  $L_-\to (-\infty,0]\times\Lambda$  which is defined by sending the ray $(-\infty,0]\times x$, $x\in\Lambda$, onto the trajectory of $-S$ originated at $x\in\Lambda$.
 
Alternatively, the  negative end structure can be viewed as a {\it  negative completion} of the manifold $L_+$ with boundary $\Lambda$:
$$L=L_+\mathop{\cup}\limits_{0\times \Lambda\ni (0,x)  \sim x\in \Lambda } (-\infty,0]\times \Lambda.$$ 
 
 Negative end  structures which differ  by a  choice of the cross-section $\Lambda$ 
 transversely intersecting all  the negative trajectories of $L$ will be  viewed as  equivalent.

Let $(X,\omega)$ be  a $2n$-dimensional {\it symplectic} manifold. 
A properly embedded co-oriented hypersurface $Y\subset X$ is called a  {\it  contact slice}
if it  divides $X$ into two domains $X=X_-\cup X_+$, $X_-\cap X_+=Y$, and
 there exists a Liouville vector field $Z$ in a neighborhood of $Y$ which is transverse to $Y$, defines its given co-orientation and points into $X_+$. Such hypersurfaces are also called {\it symplectically convex}  \cite{EliGro-convex}, or of {\it contact type}  \cite{Weinstein}. 
 
 If the Liouville field extends to  $X_-$  as a non-vanishing Liouville field   such  that the negative flow $Z^{-t}$ is defined for all $t\geq 0$ and all its trajectories in $X_-$ intersect $Y$ then $X_-$  with a choice of such  $Z$  is called a {\it negative Liouville end} structure
 of the symplectic manifold $(X,\om)$.
 
 The restriction $\alpha$ of the Liouville  form $\lambda=i(Z)\om$ to $Y$ is a contact form on $Y$ 
  and   the  diffeomorphism  $(-\infty,0]\times Y\to X_-$ which sends each  ray $(-\infty,0]\times x$ onto the trajectory of $-Z$ originated at $x\in\Lambda$ is a   Liouville  isomorphism between the negative symplectization $((-\infty,0]\times Y,d(t\alpha))$ of the contact manifold 
  $(Y,\{\alpha=0\})$ and $(X_-,\lambda)$. 
  % If the contact slice $Y$ is fixed  we will view    the parameter $t$ as a coordinate on the negative end $X_-$.
   Hence     alternatively the  negative Liouville  end structure can be viewed as a {\it  negative completion} of the manifold $X_+$ with the negative contact boundary $Y$, i.e. as an
 attaching   the negative symplectization 
 $((-\infty,0]\times Y,d(t\alpha))$ of the contact manifold 
  $(Y,\{\alpha=0\})$ to  $X_+$ along $Y$.

A negative Liouville end   structure which differs by another choice of the cross-section $Y$  
 transversely intersecting all  negative trajectories of $X$   will be  viewed as an equivalent one. Note  that the holonomy along trajectories of $X$  provides a  contactomorphism between any  two transverse sections.   Any such transverse section will be called a {\it contact slice}.

If the symplectic form $\omega$ is exact and the Liouville form $\lambda$ is extended as a Liouville form, still denoted by $\lambda$, to the  whole manifold  $X$, then we will call $(X,\lambda)$ a {\it Liouville manifold  
with a negative  end}.

Let $L$ be an  $n$-dimensional manifold  with a negative   end, and $X $   a symplectic $2n$-manifold  with a negative  Liouville   end. A  proper Lagrangian immersion  $f:L\to X$ is  called {\it cylindrical at $-\infty$} if it maps the negative end  $L_-$ of $L$ into a negative end  $X_-$ of $X$, the restriction $f|_{L_-}$ is an embedding, and the  differential $df|_{TL_-}$ sends the vector field $S$ to $Z$.  Composing the restriction of $f$ to  a transverse slice $\Lambda$  with the projection of the negative Liouville end of $X$ to $Y$ along trajectories of $Z$ we get a Legendrian embedding $f_{-\infty}:\Lambda\to Y$, which  will  be called the {\it asymptotic negative  boundary} of the Lagrangian immersion $f$.
    
\subsubsection*{The action class} 
 
 Given a  proper  Lagrangian  immersion $f:L\to X$, we   consider its mapping cylinder $C_f=L\times [0,1]\mathop{\cup}\limits_{ (x,1)\sim f(x)}X$, which is homotopy equivalent to $X$,  and denote respectively by $H^2 (X,f)$ and $H^2_\infty(X,f)$ the $2$-dimensional   cohomology groups $H^2(C_f,L\times 0)$ and $H^2_\infty(C_f,L\times 0):=\lim\limits_{K\subset C_f}^{\longrightarrow}H^2(C_f\setminus K,(L\times 0)\setminus K)$, where the direct limit is taken over all compact subsets $K\subset C_f$. We denote by $r_\infty$ the restriction homomorphism $r_\infty: H^2(X,f)\to H^2_\infty(X,f)$.
 If $f$ is an embedding then  $H^2  (X,f)$  and  $H^2_\infty(X,f)$ are canonically isomorphic  to
 $H^2 (X,f(L))$ and  $H^2_\infty (X,f(L)):=\lim\limits_{K\subset X}^{\longrightarrow}H^2(X\setminus K,f(L) \setminus K)$, respectively.
We define the 
  {\it  relative action class }   
      $A(f)\in  H^2 (X,f)$
of  a  proper Lagrangian  immersion  $f:L\to X$     as the class defined by  the  closed 2-form which is equal  $\omega$ on $X$ and to $0$ on $L\times 0$. We say that $f$ is {\it weakly  exact } if
 $A(f)=0$.  The  {\it  relative action class at infinity } 
$A_\infty(f)\in  H^2_\infty (X,f)$ is defined as $A_\infty(f):=r_\infty(A_\infty)$. We note we have  $A_\infty(f)=A_\infty(g)$ if Lagrangian immersions $f,g$ coincide outside a compact set.

  Consider next a compactly supported Lagrangian regular homotopy, $f_t\colon L\to X$, $0\le t\le 1$, and write $F\colon L\times[0,1]\to X$, 
  for $F(x,t)=f_t(x)$. Let $\alpha$ denote the 1-form on $L\times[0,1]$
   defined by the equation $\alpha:=\iota_{\p/\p t}(F^*\omega)$, where 
   $t$ is the coordinate on the second factor of $L\times[0,1]$. 
   Then the restrictions $\alpha_t:=\alpha|_{L\times \{t\}}$ are closed for 
   all $t\in[0,1]$. We call the Lagrangian regular homotopy $f_t$ a 
   \emph{Hamiltonian regular homotopy} if the cohomology class
    $[\alpha_t]\in H^1(L)$ is independent of $t$. It is straightforward to 
    verify that for a Hamiltonian regular homotopy  $f_t$ the action class 
     $A(f_t)$ remains constant. Note, however, that the converse is not  necessarily true.
  
  If $X$ is a Liouville manifold, then we  define the {\it absolute
 action class} $a(f)\in H^1 (L)$ as the class of the closed form
 $f^*\lambda$, and call a Lagrangian  immersion $f$   {\it exact} if $a(f)=0$. Note that in that case
we have   $\delta(a(f))=A(f)$, where
    $\delta$ is  the boundary homomorphism $H^1 (L)\to H^2 (X,f)$ from the exact sequence of the pair $(C_f,L\times 0)$.  We will also use the notation
     $$H^1_\infty(L):=\mathop{\lim\limits^{\longrightarrow}_{K\subset L}}
     \limits_{ K \;\hbox{is compact} }H_1(L\setminus K),\; r_\infty:H^1(L)\to H^1_\infty(L),\; a_\infty(f)=r_\infty(a(f)).$$
     If the the immersion $f$ is cylindrical at $-\infty$ then the class $a_ \infty(f)$ vanishes on $L_-$. 
 
 \subsubsection*{Statement of   main theorems}
 We say that a symplectic manifold $X$ has    infinite Gromov width if  an arbitrarily large ball in $\R^{2n}_\st$ admits a  symplectic embedding     into $X$.
 For instance, a complete Liouville manifold  have infinite Gromov width. 
\begin{theorem}\label{thm:main}
Let $f:L\to X$ be a cylindrical at $-\infty$ proper embedding of an $n$-dimensional, $n\geq 3$,  connected manifold $L$, such that its asymptotic negative Legendrian boundary has a component which is loose in the complement of the other components. Suppose that there exists a compactly supported homotopy of injective homomorphisms $\Psi_t:TL\to TX$ covering $f$ and such that $\Psi_0=df$and $\Psi_1$ is a Lagrangian homomorphism. 
 If $n=3$   assume, in addition, that the manifold $X\setminus f(L)$ has infinite Gromov width.
 Then given a cohomology class $A\in H^2(X,f(L))$ with $r_\infty(A)=A_\infty(f)$,  
 there exists a compactly supported isotopy $f_t:L\to X$ such that 
 \begin{itemize}
 \item $f_0=f$;
 \item $f_1$ is Lagrangian;
 \item $A(f_1)=A $ and
 \item $df_1:TL\to TX$ is homotopic to $\Phi_1$ through Lagrangian homomorphisms.
  \end{itemize}
   If $X$ is a Liouville manifold with a negative contact end, then one can in addition prescribe any  value $a\in H^1(L) $ to the absolute action class $a(f_1)$  provided that $r_\infty(a)=a_\infty$, and in particular make the Lagrangian embedding $f_1$ exact.
\end{theorem}

We do not know whether the   infinite width condition when $n=3$ is really necessary, or it is just a result of  deficiency of our method.
  
  Suppose we are given a smooth proper immersion   $f: L^n  \to X^{2n}  $   with only transverse double points and  which is an embedding outside of a compact subset. If $L$  is connected, $L$ is orientable and $X$ is oriented and $n$ is even, we define the \emph{relative self-intersection index} of $f$, denoted $I(f)$, to be the signed count of intersection points, where the sign of an intersection $f(p^0) = f(p^1)$ is $+1$ or $-1$ depending on whether the orientation defined by $(df_{p^0}(L), df_{p^1}(L))$ agrees or disagrees with the orientation on $X$. Because $n$ is even, this sign does not depend on the ordering $(p^0, p^1)$; if $n$ is odd or $L$ is non-orientable we instead define $I(f)$ as an element of $\Z_2$. If $X$ is simply connected a theorem of Whitney ~\cite{Whitney} implies that $f$ is regularly homotopic with compact support to an embedding if and only if $I(f) = 0$.

  Theorem  \ref{thm:main} will be deduced in Section \ref{sec:proofs} from the following
  \begin{theorem} \label{thm:main-imm}
  Let $(X,\lambda)$ be a simply connected  Liouville manifold with a negative  end $X_-$, and
 $f:L\to X$  a cylindrical at $-\infty$ exact self-transverse Lagrangian  immersion with finitely many self intersections.    Suppose that    $I(f)=0$, and  the asymptotic negative boundary $\Lambda$ of $f$ has a component which is loose in the complement of the others.  If $n=3$   suppose, in addition, that $X\setminus f(L)$ has infinite Gromov width.
 Then there exists a compactly supported Hamiltonian regular homotopy  $ f_t$, connecting $f_0=f$ with an embedding $f_1$.  
   \end{theorem}
   \begin{remark*} If $X$ is not simply connected the statement remains true if the self-intersection index $I(f)$ is understood as an element of the group ring of $\pi_1(X)$.
   \end{remark*}
  \section{Weinstein recollections and other preliminaries}\label{sec:Weinstein-recollect}
       % \section{Surgeries of Liouville domains}
       \subsubsection*{Weinstein cobordisms}
We  define below a slightly more general notion of a  Weinstein cobordism than is usually done (comp. \cite{CieEli-Stein}), by allowing cobordisms between non-compact manifolds. Let $W$ be a $2n$-dimensional smooth manifold with boundary. We allow $W$, as well as its boundary components to be non-compact.
Suppose that the boundary $\p W$ is presented as the union of two disjoint subsets $\p_\pm W$ which are open and closed in $\p W$.
A {\it Weinstein cobordism } structure on $W$  is a triple $(\om, Z,\phi)$, where $\om$ is a symplectic form on $W$, $Z$ is a Liouville vector field, and $\phi:W\to
 [m,M]$ a Morse function with finitely many critical points, such that 
 \begin{itemize}
 \item  $\p_-W=\{\phi=m\}$ and $\p_+W=\{\phi=M\}$ are regular level sets;
 \item  the vector field $Z$ is gradient like for $\phi$, see \cite{CieEli-Stein}, Section 9.3;
 \item outside a compact subset of $W$ every trajectory of $Z$ intersects both $\p_-W$ and $\p_+W$.
 \end{itemize}
  
  The function $\phi$ is called a {\it Lyapunov function} for $Z$.
 The Liouville form $\lambda=i(Z)\om$ induces contact structure  on  all regular levels of the function  $\phi$.
 All  $Z$-stable manifolds of critical points of the function $\phi$ are isotropic for $\om$  and, in particular, indices  of all critical points are $\leq n=\frac{\dim W}2$.  
    A Weinstein cobordism  $(W,\om,X,\phi)$ is called {\it subcritical} if indices of all critical points are $<n$.

 \subsubsection*{Extension of Weinstein structure}
 The following lemma is the standard handle attaching statement in the Weinstein category (see \cite{Weinstein} and   \cite{CieEli-Stein}). We provide a proof here because we need it in a slightly different  than it is presented   in   \cite{Weinstein} and   \cite{CieEli-Stein}.

 \begin{lemma}\label{lm:surgery}
  Let $(X,\lambda)$ be a Liouville manifold with boundary,   $Z$   the Liouville field corresponding to $\lambda$ (i.e. $\iota_Z\om=\lambda$ where $\om=d\lambda$) and   $Y\subset \p X$ a  (union of)   boundary component(s)  of $X$ such that $Z$ is inward transverse to $Y$.
Let $(\Delta,\p \Delta)\subset (X,Y)$ be a $k$-dimensional ($k\leq n$)  isotropic disc, which is tangent to   $Z$ near $\p\Delta$.
  If $k=1$   suppose, in addition, that  $\int\limits_\Delta \lambda=0$, and if $k<n$   suppose, in addition, that $\Delta$ is extended to (a germ of) a Lagrangian submanifold   $(L,\p L)\subset  (X,Y)$ which is   also tangent to $Z$ near $\p L$.  Then  for   any neighborhoods  $U \supset\Delta$ and $\Omega\supset Y$   there exists  a Weinstein  cobordism $(W,\om, \wt Z,\phi)$ with the following properties :
     \begin{itemize}
  \item $Y\cup\Delta\subset W\subset\Omega\cup U$;
  \item $\p_-W=Y$;
\item the function $\phi$ has a unique  
critical point $p$  of index $k$ at the center of the disc $\Delta$;
\item the disc $\Delta$ is contained in the $\wt Z$-stable manifold of the point $p$;
  \item the field $\wt Z|_{L\cap W}$ is tangent to $L$;
  \item the Liouville form $\wt \lambda=i(\wt Z)\om$ can be written as $\lambda+dH$ for a function $H$ compactly supported in $U\setminus Y$.
    \end{itemize}
 \end{lemma}
 \begin{proof}
 Let us set $L=\Delta$ if $k=n$. For a general case we can assume that $L=\Delta\times \R^{n-k}$. Let $\om_\st$ denote the symplectic form on $T^*(L)=T^*L\times T^*\R^k =\Delta^k\times\R^k\times\R^{n-k}\times\R^{n-k}$ given
 by the formula
 $$\om_\st=\sum\limits_1^k dp_i\wedge dq_i+\sum\limits_1^{n-k} d u_j\wedge dv_j$$ with respect to the coordinates
 $(q,p,v,u)\in \Delta^k\times\R^k\times\R^{n-k}\times\R^{n-k}$ which correspond to this splitting.
  Denote by $\lambda_k$ the Liouville form
  $\sum\limits_1^k(2 p_i dq_i+q_i dp_i)+\frac12\sum\limits_1^{n-k} (v_idu_j-u_jdv_j)$, $d\lambda_k=\om_\st$.
  Note that the    Liouville field $$Z_k:=
  \sum\limits_1^k\left(-q_i \frac{\p}{\p q_i}+2 p_i \frac{\p}{\p p_i}\right)+\frac12\sum\limits_1^{n-k} \left(v_i\frac{\p}{\p v_i}+u_j\frac{\p}{\p u_j}\right)$$ corresponding to the form $\lambda_k$ is gradient like for the quadratic function
 $$Q:=\sum\limits_1^k (p_i^2-q_i^2)+\sum\limits_i^{n-k}(u_j^2+v_j^2),$$
 tangent to $L$, and     the disc $\Delta$  serves as the $Z_k$-stable manifold of its critical point.
 
Using  the     normal form for the Liouville form  $\lambda$ near $\p L$ (see \cite{Weinstein}, and also \cite{CieEli-Stein}, Proposition 6.6) and the Weinstein  symplectic normal form along the Lagrangian $L$   we can find, possibly decreasing the neighborhoods $\Omega$ and $U$,
a
symplectomorphism $\Phi:U\to U'$, where $U'$  is a neighborhood of $\Delta$ in $ T^*L$,    such that
 \begin{itemize}
\item $\Phi(L\cap U )=L \cap U'$, 
 $\Phi(\Delta\cap U)=\Delta \cap U'$;
 \item $\Phi^*\om_\st=\om $;
\item $\Phi^*\lambda_k =\lambda$ on $ \Omega\cap U$;
 \item $\Phi (Y\cap U)=\{Q =-1\}\cap U'.$
 \end{itemize}
 Thus the closed, and hence exact $1$-form  $\Phi_*\lambda-\lambda_k$ vanishes on $\Omega':=\Phi(\Omega\cap U)$, and therefore, using the condition $\int\limits_\Delta\lambda=0$  when $k=1$, we  can conclude that $\lambda_k=\Phi_*\lambda  +  d H$ for a function  $ H:G\to\R$ vanishing on $\Omega'\supset\p\Delta$. Let $\theta:U'\to[0,1]$ be a $C^\infty$-cut-off function equal to $0$ outside a neighborhood $U_1'\supset\Delta$, $U'_1\Subset U'$, and equal to $1$ on a smaller neighborhood   $U_2'\supset\Delta$, $U'_2\Subset U'_1$. Denote $\wh H:=\theta H$. Then the form
 $\wh\lambda:=\Phi_*\lambda+d\wh H$   coincides with $\Phi^*\lambda$ on $\Omega'\cup (U'\setminus U_1')$, and  equal to $\lambda_k$ on $U_2'$.
 
  Then, according to Corollary 9.21 from \cite{CieEli-Stein}, for any sufficiently small $\eps>0$ and a neighborhood $U_3'\supset\Delta$, $U_3'\Subset U_2'$, there exists a Morse function $\wh Q: U'\to\R$ such that
  \begin{itemize}
  \item $\wh Q$ coincides with $Q$ on $\{Q\leq-1\}\cup(\{Q\leq-1+\eps\}\setminus U_2'$;
  \item $\wh Q$ and $Q$ are target equivalent over $U_3'$, i.e. 
  there  exists a diffeomorphism $\sigma:\R\to\R$  such that  over $U_3'$ we have 
    $\wh Q =\sigma\circ Q$;
  \item $-1+\eps$ is a regular value of $\wh Q$ and  $\{\wh Q\leq-1+\eps\}\subset \Omega'\cup U_2'$;
  \item   inside
  $\wh W:=\{-1\leq \wh Q\leq -1+\eps\}\subset U'$ the function $\wh Q$ has a unique critical point.
  \end{itemize}
      
     Denote       $\wt Q:=\wh Q\circ\Phi:U\to\R$. Let us extend  the function $\wt Q$ to the whole manifold  $X$ in such a way that
     \begin{itemize}
     \item $\{\wt Q=-1\}\setminus U=Y\setminus U$, 
   \item  $\{-1\leq \wt Q\leq -1+\eps\}\setminus U\subset \Omega\setminus U$, 
   \item  the function $\wt Q|_{X\setminus U}$ has no critical values in $[-1,-1+\eps]$ and
   \item  the  Liouville vector field $Z$ is gradient like for $\wh Q$ on
   $\{-1\leq \wt Q\leq -1+\eps\}\setminus U$.
   \end{itemize}
   Let us define $W:=\{-1\leq\wt Q\leq -1+\eps\}\subset X$,
   $$\wt\lambda=\begin{cases}
   \Phi^*\wh\lambda=\lambda+d\wh H\circ \Phi,& \hbox{on}\; U,\cr
   \lambda,&\hbox{on}\; X\setminus U.
   \end{cases}
   $$
   Let $\wt Z$ be the Liouville field $\omega$-dual to the  Liouville form $\wt\lambda$
   Then the Weinstein cobordism  $(W,\om, \wt Z,\phi:=\wh H\circ\Phi)$ has the required properties.
    \end{proof}
    
 We will also need the following simple
\begin{lemma}\label{lm:surgery-index0} Let $(X,\lambda)$ be a Liouville manifold and $f:L\to X$ a Lagrangian immersion. Let $p\in X$ be a transverse self-intersection point. Then there exists a symplectic  embedding $h:B\to  X$ of a sufficiently small ball in $\R^{2n}_\st$   into $X$ such that $h(0)=p$ and $h^{-1}(f(L))=B\cap(\{x=0\}\cup\{y=0\})$. \end{lemma}
\begin{proof} By the Weinstein neighborhood theorem, there exist coordinates in a symplectic ball near $p$ so that $f(L)$ is given by $\{x = 0\} \cup \{y = dg(x) \}$ for some function $g:\R^n \to \R$ so that $dg(0) = 0$ (here we use natural coordinates on $T^*\R^n$). By transversaility the critical point of $g$ at $0$ is non-degenerate. Composing with the symplectomorphism $(x, y) \mapsto (x, y - dg(x))$ gives the desired coordinates.
\end{proof}
\subsubsection*{Cancellation of critical points in a Weinstein cobordism}

The following   proposition  concerning cancellations of critical points 
in a Weinstein cobordism 
 is proven in \cite{CieEli-Stein}, see there
  Proposition 12.22. 

\begin{prop}\label{prop:cancellation}
Let $(W,\om,Z_0,\phi_0)$ be a Weinstein cobordism with exactly two
critical points $p,q$ of index $k$ and $k-1$, respectively, which are
connected by a unique $Z$-trajectory along
which the stable and unstable manifolds intersect transversely. 
Let $\Delta$ be   the closure of the stable
manifold of the critical point $p$. Then there exists a
Weinstein cobordism structure $(\om,Z_1,\phi_1)$ with the following properties:
\begin{enumerate}
\item   $(Z_1,\phi_1)=(Z_0,\phi_0)$ near $\p W$
and outside a neighborhood of $\Delta$; 
\item  $\phi_1$ has no
critical points. 
\end{enumerate}
\end{prop}

\subsubsection*{From Legendrian isotopy to Lagrangian concordance}
The following Lemma  about Lagrangian realization of a Legendrain isotopy is proven in \cite{ELIGRO-findim}, see there Lemma 4.2.5.
 \begin{lemma}\label{lm:Leg-Lag}
Let $f_t:\Lambda\to (Y, \xi=\{\alpha=0\}) $, $t\in[0,1]$, be a Legendrian isotopy connecting $f_0,f_1$. Let us extend it to $t\in\R$ as independent of $t$ for $t\notin[0,1]$. Then   there exists  a Lagrangian embedding
 $$F:\R\times\Lambda\to  \R\times Y, d(e^s\alpha)),$$  of the
 form  $F (t,x)=(\wt f_t(x),h(t,x))$  such that
 \begin{itemize}
 \item $F(t,x)=(f_1(x),t)$ and $F(x,-t)=f_0(x)$ for  $t>C$, for a sufficiently large constant $C$;
 \item $\wt f_t (x)\;$  $C^\infty$-approximate $f_t(x)$.
 \end{itemize}
 \end{lemma}

\section{Action-balanced Lagrangian immersions}\label{sec:balanced}
 
Suppose we are given an exact proper Lagrangian immersion $f:L\to X$  of an orientable manifold $L$
into a simply connected Liouville manifold $(X,\lambda)$ with  finitely many transverse self-intersection points.  For each 
 self-intersection point $p\in X$ we denote by $p^0,p^1\in L$ its pre-images in $L$. 
The integral $a_\SI(p,f)=\int\limits_\gamma f^*\lambda$, where $\gamma:[0,1]\to L$ is any path connecting the points  $\gamma(0)= p^0$ and $\gamma(1)= p^1$, will be called the {\it action} of the self-intersection point $p$. Of course, the sign of the action depends on the ordering  of the pre-images $p^0$ and $p^1$. We will fix this ambiguity by requiring that $a_\SI(p,f) > 0$ (by a generic perturbation of $f$ we can assume there are no points $p$ with $a_\SI(p, f) = 0$).

  A pair of self-intersection points $(p,q)$ is called a {\it balanced  Whitney pair}  
   if  $a_\SI(p,f)=a_\SI(q,f)$ and  the intersection indices  of  $df(T_{p^0}L)$ with $df(T_{p^1}L)$  and  of $df(T_{q^0}L)$  with  $df(T_{q^1}L)$  have opposite signs. In the case where $L$ is non-orientable we only require that $p$ and $q$ have the same action. A Lagrangian immersion $f$ is called {\it  balanced} 
if the set of its self-intersection points can be presented as the union of disjoint balanced Whiney pairs.

 The goal of this section is the following
  \begin{prop}\label{prop:balancing} Let $(X,\lambda)$ be a simply connected  Liouville manifold  with a negative  end and      $f:L\to X$   a proper exact and cylindrical at $-\infty$ Lagrangian immersion with finitely many  transverse double points.  If $n=3$ suppose, in addition, that $X\setminus f(L)$ has infinite Gromov width.
        Then there exists an exact cylindrical at $-\infty$ Lagrangian regular homotopy $f_t:L\to X $, $t\in[0,1]$, which is  compactly supported away from the negative end, and such that
   $ f_0=f$ and      $f_1$ is balanced. \\
   If the asymptotic negative boundary of $f$ has a component which is loose in the complement of the other components and $I(f) = 0$ then the Lagrangian  regular homotopy $f_t$ can be made fixed at $-\infty$.
  \end{prop}
 Note that Proposition \ref{prop:balancing} is the only step in the proof of the main results of this paper where one need the infinite Gromov width condition  when $n=3$.

The following two lemmas will be used to  reduce the action of our intersection points in the case where we only have a finite amount of space to work with, for example when $X_+$ is compact. In the case where $X_+$ contains a symplectic ball $B_R$ of arbitrarily large radius, e.g. in the situation of Theorem  \ref{thm:caps}, these lemmas are not needed. 

                  \begin{lemma}\label{lm:small-action-model}
             Consider an annulus $A:=[0,1]\times S^{n-1} $.  Let $x,z$ be coordinates corresponding to the splitting, and  $y,u$ the dual coordinates in the  cotangent bundle
        $ T^*A$, so that the canonical Liouville form $\lambda$ on $T^*A$ is equal to $ydx+udz$.
        Then for any   integer $N>0$  there exists a Lagrangian immersion $\Delta:A\to T^*A$ with the following properties:
          \begin{itemize}
           \item   $\Delta(A)\subset \{ |y|\leq \frac5N, ||u|| \leq\frac5N\}$;
          \item  $\Delta$ coincides with the inclusion of the zero section $j_A:A\hookrightarrow T^*A$ near $\p A$;
          \item there exists a fixed near $\p A$ Lagrangian regular homotopy connecting
          $j_A$ and $\Delta$;
      \item  $ \int\limits_{\zeta}\lambda=1$, where
      $\zeta$ is  the  $\Delta$-image of any  path connecting $S^{n-1}\times 0$ and $S^{n-1}\times 1$ in $A$;
      \item action of any self-intersection point of $\Delta$ is $<\frac1N$;
      \item the number of self-intersection points is $<8N^3$.
                 \end{itemize}   
                 \end{lemma}
                 \begin{proof}
                  Consider in $\R^2$  with coordinates $(x,y)$ the   rectangulars $$I_{j,N}=\left\{\frac{j}{5N^2}\leq x \leq\frac{j}{5N^2}+\frac1{5N},0\leq y\leq\frac5{N}\right\}, j=0, \dots (N-1)N.$$
       
      Consider   a path $\gamma$ in $\R^2$ which begins at the origin, travels  counter-clockwise   along the boundary of $I_{0, N}$, then moves along the $x$-axis to the point $(\frac1{5N^2},0)$, travels  counter-clockwise   along the boundary of $J_{1, N}$   etc., and ends
      at the point $(1,0)$.
      Note that $\int\limits_\gamma ydx=\frac{N-1}{N}$.  We also observe that   squares $I_{j,N}$  and $I_{i,N}$ intersect  only  when  $|i-j|\leq N$, and hence for any self-intersection point $p$ of $\gamma$ its action  is bounded by $ N\frac1{N^2}=\frac1N.$ 
      Let  us  $C^\infty$-approximate $\gamma$ by an immersed curve $ \gamma_1$ with transverse self-intersections and which coincides with $\gamma$ near its   end points.  We can arrange that 
      \begin{itemize}
      \item  $\left|\int\limits_{ \gamma_1}ydx-1\right|<\frac2N$;
      \item action of any self-intersection point of $ \gamma_1$ is $<\frac1N$;
      \item the number of self-intersection points is $<2N^3$;
      \item the  curve $ \gamma_1$  is contained in    the rectangular
      $\{0\leq x\leq \frac15,0\leq y\leq\frac5N\}$. 
      \end{itemize}
      
      \begin{figure}
 \center{\includegraphics[scale = .8]{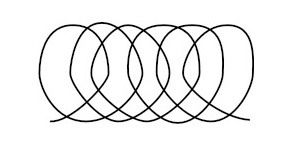}} \caption{The curve $\gamma_1$ when $N=3$.} \label{fig:twirl}
\end{figure}

      See Figure ~\ref{fig:twirl}. The only non-trivial statement is the upper bound on the number of self-intersections. Notice that there are less than $N^2$ loops, and each loop intersects at most $2N$ other loops, in $2$ points each. Thus the number of self intersections, double counted, is less than $4N^3$.
      
      We will assume that $ \gamma_1$ is parameterized by the interval $[0,\frac15]$.
       Let  $r_N$  denote  the affine   map  $(x,y)\mapsto (x+\frac15,-\frac yN)$. We
        define  a path $\gamma_2:[\frac15,\frac25]\to\R^2$ by the formula
        $$\gamma_2(  t)=r_N(\gamma_1(t-\frac15) ).$$    
        Note that the immersion
        $\gamma_{12}:[0,\frac25]\to\R^2$ which coincides with $\gamma_1$ on 
        $[0,\frac15]$ and with $\gamma_2$  on $[ \frac15,\frac25]$ is regularly homotopic to the straight interval  embedding via a homotopy  which is fixed near the end of the interval, and which is inside $\{0\leq x\leq \frac25,-\frac5{N^2} \leq y\leq\frac{5}{N}\}$. We also note that
         $\left|\int\limits_{\gamma_{12}}ydx-1\right|<\frac3N$. See Figure ~\ref{fig:twirl2}.
         
               \begin{figure}
 \center{\includegraphics[height=50mm]{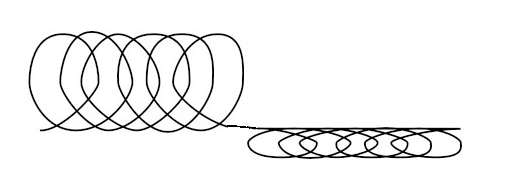}} \caption{The curve $\gamma_{12}$.} \label{fig:twirl2}
\end{figure}

        We further extend $\gamma_{12}$ to an immersion
        $\gamma_{123}:[0,1]\to\R^2$ by extending it to $[\frac25,1]$ as a graph of function  $\theta:[\frac25,1]\to[-\frac5N,\frac5N]$ with   $$\int\limits_{2/5}^1\theta(x)dx=1-\int\limits_{\gamma_{12}}ydx,$$ 
        which implies $\int\limits_{\gamma_{123}}ydx = 1$.
        
       Let $j_{S^{n-1}}$ denote the inclusion $S^{n-1}\to T^* S^{n-1}$ as the $0$-section. Consider a Lagrangian immersion  $\Gamma:A \to T^*A$ given by 
        the formula
        $$\Gamma(x,z)=(   \gamma_{123}(x),j_{S^{2n-1}}(z))\in T^*[0,1]\times T^*S^{n-1}=T^*A.$$
        The Lagrangian immersion $\Gamma$ self-intersects along spheres of the form  $p\times S^{n-1}$ where $p$  is a self-intersection  point of $\wt \gamma$. By a $C^\infty$-perturbation of  $\Gamma$ we can construct  a Lagrangian immersion $\Delta:A\to T^*A$ with transverse self-intersection points which have all the properties listed in Lemma \ref{lm:small-action-model}. Indeed, for each of the $4N^3$ intersection points $p$ of $\gamma_{123}$, the sphere $p \times S^{n-1}$ can be perturbed to have two self-intersections. The other required properties are straightforward from the construction.
                            \end{proof}

  \begin{remark}\label{rm:scaling}
          {\rm
      Given any $a>0$ we get,    by scaling  the Lagrangian immersion $\Delta$ with the dilatation $(y,u)\mapsto (ay,au)$,  a Lagrangian immersion $\Delta_a:A\to T^*A$ which satisfy
       \begin{itemize}
      \item  $ \int\limits_{\zeta}\lambda =a $, where
      $\zeta$ is  the  $\Delta_a$-image of any  path connecting the boundary  $S^{n-1}\times 0$ and $S^{n-1}\times 1$ of $A$;
      \item action of any self-intersection point of $\Delta_a$ is $<\frac aN$;
      \item the number of self-intersection points is $<8N^3$;
      \item   $\Delta_a(A)\subset \{|y|,||u|| \leq\frac{5a}N\}$;
      \item the immersion $\Delta_a$ is regularly homotopic relative its boundary to the inclusion $A\hookrightarrow T^*A$.
            \end{itemize} 
}
          \end{remark}
                         Given a proper Lagrangian immersion $f:L\to X$  with finitely many transverse self-intersection points, we denote the number of self-intersection points by $\SI(f)$. The action of a self-intersection point $p$ of $f$ is denoted by $a_\SI(p,f)$. 
We set  $a_{\SI}(f):=\mathop{\max }
           \limits_p|a_\SI(p,f)|$, where the maximum is taken over all self-intersection points of $f$.    

           \begin{lemma}\label{lm:small-action}
           Let $f_0:L\to (X,\lambda)$ be a proper  exact Lagrangian immersion into a simply connected  Liouville manifold  with finitely many transverse self-intersection points.          Then for any sufficiently large integer  $N>0$ there exists a  fixed at infinity $C^0$-small exact Lagrangian regular homotopy $f_t:L\to X$,  $t\in[0,1]$, such that $f_1$ has transverse   self-intersections, $$a_{\SI}(f_1)\leq
           \frac{a_{\SI}(f)}N,\;\;               \SI(f_1)\leq 9N^3 \SI(f_0) . $$                     
              \end{lemma}
   \begin{proof}
  Let $p_1,\dots, p_k$ be the self-intersection points of $f_0$ and 
  $p^0_1, p^1_1,\dots, p^0_k,p^1_k$ their pre-images, $k=\SI(f_0)$.
  Let us recall that we order the pre-images in such a way that $a_{\SI}(f_0)(p_i)>0$, $i=1,\dots, k$. Choose 
  \begin{description}
  \item {-} disjoint  embedded $n$-discs $D_i\ni p^1_i$, $i=1,\dots, k$, which do not contain  any other pre-images  of  double points, and
 \item{-} annuli $A_i\subset D_i$   bounded by two concentric spheres in $D_i$.
 \end{description}
  For a sufficiently large $N>0$ there exist  disjoint symplectic embeddings $h_i$ of the domains 
  $U_i:= \{|y|,||u|| \leq\frac{5a_\SI(p,f_0)}N\}\subset T^*A$ in $X$, $i=1,\dots, k$, such that $h_i^{-1}(f_0(L))=h_i^{-1}(A_i)=A$. Then, using Remark \ref{rm:scaling},
  we find a Lagrangian regular homotopy $f_t$ supported in 
  $\bigcup\limits_1^kh_i(U_i)$ which annihilates the action of points $p_i$, i.e.
   $a_\SI(p_i,f_1)=0$, $i=1,\dots k$, and which creates no more than $8kN^3$ new self-intersection points  of action $< \frac{a_\SI(f_0)}N$. Hence, the total number of self-intersection points of $f_1$ satisfies the inequality
  $ \SI(f_1)< 9\SI(f_0)N^3$.
  
 \end{proof} 

The next lemma is a local model which will allow us to match the action of a given intersection point, during our balancing process. For a positive $C$ we denote by $Q_C$ the parallelepiped
 $$\{|z|\leq C, |x_i|\leq 1,|y_i|\leq C, \;i=1,\dots,  n-1\}$$  in the standard contact space $\R^{2n-1}_\st=(\R^{2n-1},\xi=\{\alpha_\st:=dz-\sum\limits_1^{n-1}y_idx_i=0\})$. 
 Let $SQ_C$ denote the domain 
$[\frac12,1]\times Q_C$ in the symplectization $(0,\infty)\times Q_C $  of $Q_C$ endowed with the Liouville form  $\lambda_0:=s\alpha_\st $. We furthermore denote by $L^t $  the Lagrangian rectangular $\{z=t, y=0; j=1,\dots, n-1\}\cap SQ_C\subset SQ_C$, $ t\in[-C,C]$.
%%%%%%%%%%%%%%%%%%%%%%

\begin{lemma}\label{lm:loose-trick-model}
 For any positive  $b_0,b_1,\dots, b_k\in (0, \infty) $, $ k\geq 0$, such that $$\frac{C}{4k+4}>b_0>\max(b_1,\dots, b_k),$$ and a sufficiently small $\eps>0$ there 
  exists    a  Lagrangian isotopy which starts   at $L^{-\eps}$, fixed near $1\times Q_C$ and $[\frac12,1]\times\p Q_C$, cylindrical near $\frac12\times Q_C$, and  which     ends at   a Lagrangian submanifold 
$\wt L^{-\eps}$ with the following properties:
\begin{itemize}
\item
$\wt L^{-\eps}$ intersects $L^0$ transversely at $k+1$ points  $B_0, B_1,\dots, B_k$; 
\item  if $\gamma_{B_j}, j=0.\dots, k$,   is a path in $\wt L^{-\eps}$  connecting the point    $B_j$ with a point on the boundary $\p Q_C$, then
$$  \int\limits_{\gamma_{{}_{B_j}}}\lambda_0=b_j,\; j=0,\dots, k;$$ 
\item the intersection indices of $L^0$ and $\wt L^{-\eps}$ at the points
$B_0, B_1,\dots, B_k$ are equal to $1,-1,\dots, -1$, respectively.
\item $\wt L^{-\eps}\cap \{s=\frac12\}$ is a Legendrian submanifold in $Q_C$ defined  by a generating function which is equal to $-\eps$ near $\p Q_C$ and positive over a domain in $Q_C$ of Euler characteristic $1-k$.
\end{itemize}

   \end{lemma}
   \begin{proof}
   We have
   $$\omega:=d\lambda_0=ds\wedge dz-\sum\limits_1^{n-1}dx_i \wedge d(sy_i)  =-d(zds+\sum\limits_1^{n-1}v_i dq_i),$$
    we denoted $ v_i:=sy_i, \; i=1,\dots, n-1.$ Let $I^{n-1}\subset\R^{n-1}$ be the  cube
   $\{\mathop{\max}\limits_{i=1,\dots, n-1}|q_i|\leq 1\}$.
   Choose a smooth non-negative function  $\theta:[\frac12, 1]\to \R$
   such that 
  \begin{itemize}
  \item $\theta(s)=s$ for $s\in[\frac12,\frac58]$;
  \item   $\theta$ has a unique local maximum at a point $\frac34$;
    \item $\theta(s)=0$ for $s$ near $1$;
  \item the derivative $\theta'$ is monotone decreasing on $[\frac58,\frac34]$.
  \end{itemize}
  
For any    $ \wt b_0,\dots, \wt b_k\in(0,\frac{C}{2k+2})$  which satisfy
$\wt b_0>\max(\wt b_1,\dots,\wt b_k)$ one can construct a   smooth non-negative function $\phi:I^{n-1}\to\R$. 
   with the following properties:
      \begin{itemize}
      \item $\phi= 0$ near $\p I^{n-1}$;  
      \item $ \mathop{\max}\limits_{i=1,\dots, n-1} \left|\frac{\p \phi}{\p q_i}\right|<\frac C2$;
      \item besides degenerate critical points corresponding to the critical value $0$, the function $\phi$ has $k+1$ positive non-degenerate critical points:  $1$ local  maximum $\wt B_0$ and $k$ critical points $ \wt B_1,\dots,\wt B_k$ of index $n-2$ with   critical values $ \wt b_0,  \wt b_1,\dots,  \wt b_k$ respectively.
          \end{itemize}
           
       Take a positive $\eps<\min(\wt b_1,\dots, \wt b_k, \frac{  C}{8k+8})$ and define a function 
   $h:[\frac12,1]\times I^{n-1}\to\R$ by the formula
   $$h(s,q)= -\eps s+\theta(s)\phi(q),\; s\in\left[\frac12,1\right], q\in I^{n-1}.$$

      Thus   the function  $h$  is equal to $s(-\eps+\phi(q))$ for $s\in[\frac12,\frac58]$ and equal to  $-\eps s$ near the rest of the boundary of $[\frac12,1]\times I^{n-1}$. The function $h$ has one local  maximum at a point $(s_0, \wt B_0)$ and $k$ index $n-1$ critical points with coordinates $(s_j, \wt B_j)$, $j=1,\dots, k$.
      Here the values $s_j\in[\frac58,\frac34]$ are determined from the equations
      $\wt b_j\theta'(s_j) =\eps$, $j=0,\dots, k$. 
      Respectively, the critical values  
      are equal to $ \wh b_k:= -\eps s_j+\theta(s_j)\wt b_j$,
      For  $\wt b_j$  near   $\eps$  we have  $\wh b_j<\eps $, while for $\wt b_j$ close to $\frac C{2k+2}$ we have $\wh b_j>\frac C{4k+4}$. Hence, by continuity, any critical values $    b_0,  b_1, \dots  b_k\in   \left(\eps, \frac C{4k+4}\right)$ which satisfy the inequality
      $b_0>\max(b_1,\dots, b_k)$ can be realized.
      
     The required Lagrangian manifold $\wt L^{-\eps}$  can be now defined by the  equations      $$ z =\frac{\p h}{\p s},\; x_j = q_j,\;v_j=\frac{\p h}{\p p_j}, j=1,\dots, n-1,\; s\in\left[\frac12,1\right],\; q\in I^{n-1},$$ or
      returning to $x,y,z,s$ coordinates by the equations
      $$\wt L^{-\eps}=\left\{z=  \frac{\p h}{\p s}, y_j=\frac1s\frac{\p h}{\p q_j}\right\}.$$
      It is straightforward to check that $\wt L^{-\eps}$ has the required properties.
           \end{proof}
           %%%%%%%%%%%%%%%
           
After using Lemma \ref{lm:small-action} to shrink the action of an intersection point, Lemma \ref{lm:loose-trick-model}, applied with $k=0$, will allow us to balance any negative intersection point.
 Positive intersection points still provide a challenge though, because the intersection point with the largest action created by Lemma \ref{lm:loose-trick-model} is always positive. The following lemma solves this issue.

\begin{lemma} \label{lm:2points}
 Let  $f:L\to (X,\lambda)$ be a proper     {\it exact} Lagrangian immersion  into a simply connected $X$ and $D \subset L$ an   $n$-disc  which contains no double points of the immersion $f$. Then for 
 any $A>0$  and a sufficiently small $\sigma>0$ there exists a   supported in $D $  Hamiltonian regular homotopy  of $f$ to $\wt f$ which creates a pair $p_+,p_-$ of additional   self-intersection points  such that  $a_\SI(p_\pm,\wt f)=A\pm\sigma$,  the self-intersection indices of $p_\pm$ have opposite signs and can be chosen at our will. 
\end{lemma}
Let us introduce some notation.
Consider a domain 
$$U_{\eps }:=\{-2\eps < p_1 < 1+2\eps, \mathop{\max}\limits_{1\leq i\leq n}|q_i| < 2\eps, \mathop{\max}\limits_{1\leq j\leq n}|p_j| < 2\eps \} $$
 in the standard symplectic $\R^{2n}_\st=(\R^{2n},\sum\limits_1^n dp_i\wedge dq_i)$. Let $L^t$ be the Lagrangian plane 
 $\{p_1=t,  p_j=0\;\hbox{ for}\; j=2,\dots, n \}\cap U_\eps\subset U_\eps,\; t\in\{0,1\}$. 
 Note that  $pdq|_{L^t}=tdq_1.$
We will   also use the following notation associated with $U_{\eps }$:
\begin{description}
\item{} $u_\pm\in L^1$ denote the points with coordinates
$p=(1,0,\dots, 0), q=(\pm \delta_1,0,\dots, 0)$;
\item{} $z_\pm\in L^0$  denote the points with coordinates
$p=(0,0,\dots, 0), q=(\pm \delta_1 ,0,\dots, 0)$
\item{} $c^0$ denote the point with coordinates
$p=(0,0,\dots, 0), q=(- \eps,0,\dots, 0)$;
\item{} $c^1$ denote the point with coordinates
$p=(1,0,\dots, 0), q=(- \eps,0,\dots, 0)$;
\item{}  $J^1_\pm$ denote the intervals connecting $c^1$ and $u_\pm$;
\item{} $J^0_\pm$ denote the intervals connecting $c^0$ and $z_\pm$.
\end{description}
We will use in the proof  of \ref{lm:2points}
the following
\begin{lemma}\label{lm:local-Lagrangian}

   There exists a Lagrangian isotopy   $f_t:L^1\to U_\eps$   fixed near $\p L^1$  and starting at the inclusion $f_0:L^1\hookrightarrow U_{\eps}$ such that $\wt L^1=f_1(L^1)$ transversely intersects $L^0$ at two points $z_\pm$ 
  with the following properties:
\begin{itemize}
\item $f_1^*(pdq)= q_1+d\theta$, where $\theta:L^1\to \R$ is  a compactly supported in $\Int L^1$ function such that 
$\theta(z_\pm)=\mp\delta$  for a sufficiently small $\delta>0$;
\item the intersection  indices of $\wt L^1$ and $L^0$ at $z_+$ and $z_-$ have opposite signs  and can be chosen at our will.
\end{itemize}
\end{lemma}

\begin{proof} For   sufficiently small $\delta_1,\delta_2$, $ 0<\delta_1\ll\delta_2\ll\eps$,  there exists  a $C^\infty$-function
$\alpha:[-\eps,\eps]\to\R$   with the following properties:
\begin{itemize}
\item $\alpha(t)= t$ for $\delta_2\leq |t|\leq\eps$;
\item $\alpha(t)=t^3- 3\delta_1^2 t$ for $|t|\leq \delta_1$;
\item the function $\alpha$ has no critical points, other than $\pm\delta_1$;
\item $-\frac{\eps}2<\alpha'(t)<1+ \frac{\eps}2$.
\end{itemize}

Let us also take a cut-off function $\beta:[0,1]\to[0,1]$ which is equal to $0$ near $1$ and equal to $1$ near $0$.
Take  a quadratic form $Q_j$ of index $j-1$:
 $$Q_j(q_2,\dots, q_n)= -\sum\limits_{i=2}^{j} q_i^2+\sum\limits_{j+1}^n q_i^2,  \;j=1,\dots, n,$$  
and define a function $  \sigma:\{ |q_i|\leq\eps; i=1,\dots, n\}\to\R$ by the formula
$$\sigma_j(q_1,q_2,\dots, q_n)= q_1 +\delta_2 Q_j(q_2,\dots, q_n)\beta\left(\frac{\rho}\eps\right)\beta\left(\frac{|q_1|}\eps\right)+(\alpha(q_1)- q_1)\beta\left(\frac{\rho}{\eps}\right),$$ where we denoted $\rho:=\mathop{\max}\limits_{2\leq i\leq n}|q_i|$.
The function $\sigma_j$ has two critical points $(-\delta_1,0,\dots,0)$ and $(\delta_1,0,\dots,0)$ of index $j$ and $j-1$, respectively.
We note that 
$$- \frac\eps2-  Cn\delta_2   \eps  \leq \frac{\p\sigma_j}{\p q_1} <1+\frac\eps2+ Cn\delta_2\eps$$ and 
 $$\left|\frac{\p\sigma_j}{\p q_i}\right|\leq 2\delta_2\eps+ Cn\delta_2\eps+  \frac{C\delta_2}\eps $$ for $i>1$, where $C=||\beta||_{C^1}$. In particular, if $\delta_2$ is chosen   small enough we get  $-\eps< \frac{\p\sigma_j}{\p q_1} <1+\eps$ and $\left|\frac{\p\sigma_j}{\p q_i}\right|<\eps$ for $i=2,\dots,n$.

Assuming that $L^1$ is parameterized by the $q$-coordinates we    define the required  Lagrangian isotopy $f_t:L^1\to U_\eps $
  by the formula:
  $$f_t(q)=\left(q,1+t\left(\frac{\p\sigma_j}{\p q_1}-1\right),t\frac{\p\sigma_j}{\p q_2}, \dots, t\frac{\p\sigma_j}{\p q_n}\right)),  |q_i| < 2\eps;\; i=1,\dots, n. $$
  The Lagrangian manifold $\wt L^1=f_1(L^1)$ intersects $L^0$ at two points $z_\pm$ with coordinates $p=0, q_1=\pm\delta_1,q_2=0,\dots,q_n=0$. The intersection index of  $\wt L^1$ and $L^0$ at $z_-$    is equal to $(-1)^j$, and to  $(-1)^{j-1}$ at $z_+$. Thus by choosing $j$ even or odd we can arrange the intersection to be positive at $z_+$ and negative at $z_-$, or the other way around.
  The compactly supported function $\theta$ determined from the equation $f_1^*(pdq)= dq_1+d\theta$ is equal to $\sigma_j-q_1$. In particular, $\theta(z_\pm)=\mp {2\delta_1^3} $. 
   \end{proof}

    \begin{proof}[Proof of Lemma \ref{lm:2points}]
     We  denote  $\wt J^1_\pm:=f_1(J^1_\pm)$, where $f_t$ is the isotopy constructed in Lemma \ref{lm:local-Lagrangian}. Take any two points $a,b\in D\subset \wt D:=f(D)\subset \wt L:=f(L)$ and connect them by a path
 $\eta:[0,1]\to \wt D$ such that $\eta(0)=\wt b:=f(b)$ and $\eta(1)=\wt a:=f(a)$. Denote
 $B:=\int\limits_\eta\lambda$. 
 
 For any  real $R $  there exists an embedded path $\gamma:[0,1]\to X$ connecting the points $\gamma(0)=\wt a $ and
 $\gamma(1)=\wt b $ in the complement of $\wt L $, homotopic to a path in $\wt L$ with fixed ends, and such that $\int\limits_\gamma \lambda=R $.  For a sufficiently small $\eps>0$ the embedding $\gamma$ can be extended to a symplectic embedding $\Gamma:U_{\eps}\to X$ such that $\Gamma^{-1}(\wt L)=L^0\cup L^1$. Here we identify the domain $[0,1]$ of the path $\gamma$ with the interval
$$I=\{q_1=-\eps, q_j=0, j=2,\dots, n;0\leq p_1\leq 1,p_j=0, j=2,\dots, n\} \subset \p U_{\eps },$$ so that we have $  \Gamma(c^0)=\wt a$ and $\Gamma(c^1)=\wt b$.

      \begin{figure}
 \center{\includegraphics[height=80mm]{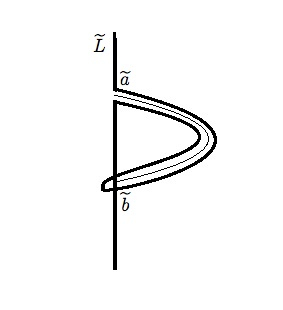}} \caption{The Lagrangian $f_1(L)$. The light curve represents $\gamma$.} \label{fig:curve}
\end{figure}

 The  Lagrangian isotopy $\wt f_t:=\Gamma\circ f_t:L^1\to X$, where  $f_t:L^1\to U_{ \eps }$ is the isotopy    constructed in Lemma  \ref{lm:local-Lagrangian},  extends as a constant homotopy to the rest of $L$  and  provides us with a regular Lagrangian homotopy   connecting the immersion $f$ with  a Lagrangian immersion $L\to X$ which has two more  transverse intersection points $p_\pm:=\Gamma(z_\pm)$ of opposite intersection index sign. See Figure ~\ref{fig:curve}.  Consider the following   loops $\zeta_\pm$ in $\wt L\subset X$ based at  the points
 $p_\pm$.  
 We start from the point  $p_\pm$ along the $\Gamma$-image   of  the oppositely oriented interval $ \wt J^1_\pm$ to the point $\wt b$ , then follow the path $\eta$ to the point $ \wt a$,  and finally follow
 along the $\Gamma$-image   of  the path $  J_0$  back to $p_\pm$.  
 
 Then we have 
 \begin{align*}
 &\int\limits_{\zeta_\pm}\lambda   =-\int\limits_{\wt J^1_\pm}\Gamma^* \lambda+\int\limits_{\eta}\lambda+\int\limits_{J^0_\pm}\Gamma^*\lambda\cr
 &=\left(-\int\limits_{\wt J^1_\pm}\Gamma^* \lambda+\int\limits_\gamma\lambda+ \int\limits_{J^0_\pm}\Gamma^*\lambda\right)+\left(\int\limits_{\eta}\lambda
 -\int\limits_\gamma\lambda\right)\cr 
 &
 =\left( -\int\limits_{\wt J_\pm^1}pdq-\int\limits_Ipdq+ \int\limits_{J^0_\pm}pdq\right)+(B+R)= -\eps+B+R \mp 2\delta_1^3.
 \end{align*}
 It remains to observe that there exists a sufficiently small  $\eps_0>0$ which can be chosen for any $R\in [A-C-1, A-C+1]$. Hence, by setting $R=A-C-\eps_0$ and $\eps=\eps_0$ we arrange  that the action of the intersection points $p_\pm$  is equal  to $A\mp2\delta_1^3$
 while their  intersection indices have opposite sign which could be   chosen at our will.
 \end{proof}

      \medskip
        \begin{lemma}     \label{lm:parallel-slices}
        Let $((0,\infty) \times Y, d(t\alpha))$ be the symplectization of a manifold $Y$ with a contact form $\alpha$. Let $\Lambda$ be a Legendrian submanifold and $L=(0,\infty)\times \Lambda$ the Lagrangian cylinder over it. Suppose that there exists a contact form preserving embedding $\Phi:(Q_C,\alpha_\st)\to (Y,\alpha)$ and $\Gamma\subset Y$ an embedded isotropic arc connecting a point  $b\in\Lambda$ with a point $$\Phi(x_1=1,x_2=0,\dots, x_{n-1}=0,  y_1=0, \dots, y_n=0, z=0)\in \p\Phi(Q_C).$$
        Then there exists a Lagrangian isotopy $L_t\subset \R\times \Lambda$ supported in a neighborhood of $1\times \Gamma\cup \Phi(Q_C)$, $t\in[0,1]$, which begins at $L_0=L$ such that
        \begin{itemize} 
        \item $L_t$ transversely intersects $1\times Y$ along a Legendrian submanifold $\Lambda_t$;
        \item $\Phi^{-1}(\Lambda_1)= \Lambda^0\cup\Lambda^{-\eps}$ for a sufficiently small $\eps>0$.
        \end{itemize}
        \end{lemma} 
                                          
       \begin{proof}
       We use below the notation $I^k_a$, $a>0$ for the cube
       $\{|x_i|\leq a, i=1,\dots, k\}\subset\R^k$.
    The embedding $\Phi$ can be extended to a slightly bigger domain $\wh Q=\{|x_i|\leq 1+\sigma,|y_i|\leq C,i=1,\dots, n-1,|z|\leq C+\sigma\}\subset\R^{2n-1}_\st$ for a sufficiently small $\sigma>0$.
         The intersection $ \wh Q\cap(\R^{n-1}=\{y=0,z=0\})$ is the cube $I^{n-1}_{1+\sigma}\subset\R^{n-1}$.
        We can assume that the  intersection of  the path $\Gamma$  with $\wh Q$ coincides with the interval $\{1\leq x_1\leq 1+\sigma, x_j=0,j=2,\dots, n-1\}\subset I^{n-1}_{1+\sigma}$.
        The Legendrian embedding $\Psi:=\Phi|_{I^{n-1}_{1+\sigma}}:I^{n-1}_{1+\sigma}\to Y$ can be extended to a bigger parallelepiped $$\Sigma=\{-1-\sigma\leq x_1\leq 2+\sigma, |x_j|\leq 1+\sigma, j=2,\dots, n-1\}\subset\R^{n-1}$$  such that the extended Legendrian embedding, still denoted by  $\Psi$,
        has the following properties:
        \begin{itemize}
        \item $\Psi(\{1\leq x_1\leq 2, x_j=0,j=2,\dots, n-1\})=\Gamma$;
        \item $\Psi(\{x_1=2\})\subset\Lambda$. 
        \end{itemize}
      For a sufficiently small  positive  $\delta<C$   the Legendrian embedding can be further extended as a contact form preserving  embedding  
       $$\wh \Psi:(\wh P:=\{(x,y,z)\in\R^{2n-1}_\st; x\in \Sigma, |y_i|\leq \delta, i=1,\dots, n-1,|z|\leq \delta\},\alpha_\st)\to (Y,\alpha),$$
    such that 
    \begin{itemize}
    \item $\wh\Psi|_{\wh P\cap \wh Q}=\Phi|_{\wh P\cap \wh Q}$;
    \item  the Legendrian manifold $\wh\Lambda:=\wh\Psi^{-1}(\Lambda)$ is given by the formulas
             $$\wh\Lambda:=\{z=\pm (x_1-2)^{\frac32}, y_1=\pm\frac32\sqrt{x_1-2},
    x_1\geq 2, y_j=0,j=2,\dots, n-1\}$$
    \end{itemize}
    (note that any point on any Legendrian admits coordinates describing $\wh \Lambda$ as above).

    Consider a cut-off  $C^\infty$-function  $\theta:[0,1+\sigma]\to[0,1]$   such that
     $\theta(u)=1$   if $u\leq 1$, $\theta(u)=0$ if $u>1+\frac\sigma2 ,
     \theta'\leq 0$,
   and denote  $$\Theta(u_1,\dots, u_{n-2}):=(3+\sigma) \prod \limits_1^{n-2} \theta( u_i),\; u_1,\dots, u_{n-2}\in[0,1+\sigma].$$
   
    For $s\in[0,1]$    denote
    $$\Omega_s:=\{2- s\Theta(|x_2|,\dots,|x_{n-1}|)\leq x_1\leq 2+\sigma\}\cap \Sigma\subset\R^{n-1}.$$
   We have
          $\Omega_1\supset \{-1-\sigma\leq x_1\leq2, |x_2|,\dots, |x_{n-1}|\leq 1\}\supset I^{n-1}_1$ and $\Omega_0=\{2\leq x_1\leq2+\sigma\}\cap \Sigma$.
    
         \begin{figure}
 \center{\includegraphics[scale=.6]{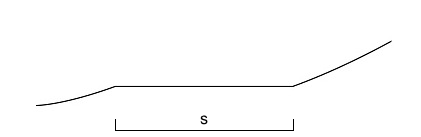} \caption{The function $g_s$.} \label{fig:function}}
\end{figure}
  
   For a sufficiently small positive $\eps<\frac{\sigma^{\frac32}}2$ consider a family of  piecewise smooth continuous functions $g_s:[2-s,2+\sigma]\to[0,\sigma^{\frac32}]$, $s\in[0,3+\sigma]$
   defined by the formulas 
   $$g_s(u)=\begin{cases}
   (u-2+s)^{\frac32},& u\leq 2-s+\eps^{\frac23};\cr
   \eps,& 2-s+\eps^{\frac23}<u<2+\eps^{\frac23};\cr
    (u-2)^{\frac32},& u\geq2+\eps^{\frac23}.
   \end{cases}
   $$
   See Figure \ref{fig:function}. We can smooth $g_s$ near the points $2+\eps^{\frac23}$ and $2-s+\eps^{\frac23}$ in such away that the derivative is monotone near these points
   (i.e. decreasing near $2-s+\eps^{\frac23}$ and increasing near 
   $2+\eps^{\frac23}$). We continue to denote the smoothened by $g_s$.
   
   Next, define for $s\in[0,1]$ a function
   $G_s:\Omega_s\to\R$ by the formula
   $$G_s(x_1,x_2\dots,x_{n-1})= g_{s\Theta(x_2,\dots, x_{n-1} )}(x_1).$$
   Note that by  decreasing $\eps$ and $\sigma$ we can arrange that $\frac{\p G_s}{\p s}(x),\left|\frac{\p G_s}{\p x_i}(x)\right|  <\delta  $, $i=1,\dots, n-1$, for all $s\in[0,1]$ and $x\in\Omega_s$.
   We also observe that   if $\frac{\p G_s}{\p x_1}(x)=0$ then $G_s(x)=\eps$. 
   Choose a cut-off function $\mu:[1-\delta,1+\delta]\to[0,1]$ which is equal to $1$ near $1$ and equal to $0$ near $1\pm \delta$ and 
      consider a  family of Lagrangian submanifolds $N_s$, $s\in[0,1]$,  defined in the domain
   $([1-\delta,1+\delta]\times \wh P, d(t\alpha_\st))$  in the symplectization of $\wh P$ defined by the formulas
  \begin{align*}
  &z=\pm  G_{s\mu(t)}(x)\pm t\frac{\p G_{s\mu(t)}}{\p t}(x) , y_i=\pm\frac{\p G_{s\mu(t)}}{\p x_i}(x),\\
  &x\in\Omega_{s\mu(t)},\; i=1,\dots, n-1,\; t\in[1-\delta,1+\delta].
  \end{align*}
 First, let us check that $N_s$ is  Lagrangian for all $s\in[0,1]$. Indeed,
 we have  $d(t\alpha_\st)=-d\left(zdt+\sum\limits_1^{n-1}(ty_i)dx_i\right)$, and hence
  $$d(t\alpha_\st)|_N=\pm d\left(\left(  G_{s\mu(t)}  +  t\frac{\p G_{s\mu(t)}}{\p t}\right)dt  
+ \sum\limits_1^{n-1}t\frac{\p G_{s\mu(t)}}{\p x_i}dx_i\right)=\pm d(d(tG_{ s\mu(t)}))=0.$$
Next, we check that $N_s$ is embedded. The only possible pairs of double points may be of the form $(x,y,z)$ and $(x,-y,-z)$, that is $z=0$ and $y=0$.
But then $\frac{\p G_{s\mu(t)}}{\p x_1}=0$, and hence $G_{s\mu(t)}(x)=\eps$ and $\frac{\p G_{s\mu(t)}}{\p t}(x)=0$, which shows 
  $z = G_{s\mu(t)}(x)+ t\frac{\p G_{s\mu(t)}}{\p t}(x)\neq 0  $. 
  
  We also note that $N_s\cap\{t=1\}$ is a Legendrian submanifold
  $\{z=  \pm G_{s\mu(t)}(x),     y_i=\pm\frac{\p G_{s\mu(t)}}{\p x_i}(x), i=1,\dots, n-1\}\subset \wh P$  and 
 $N_1$  intersects $Q_C$ along  
 $\Lambda^{-\eps}\cup\Lambda^{\eps}$.
 Near $t=1\pm\delta$ the submanifold
  $N_s$ coincides with the symplectization of the Legendrian submanifold  $\wh \Lambda$ for all $s\in[0,1]$.

 Let us remove   from  the   Lagrangian cylinder $L=(0,\infty)\times\Lambda\subset    ((0,\infty)\times Y,t\alpha)$  the domain $[1-\delta,1+\delta]\times \Lambda $ and replace it by $\Psi(N_s)$. The resulted Lagrangian isotopy $L_s$ has the following properties: $L_0=L$, $L_1$ intersects the contact slice $1\times Y$ along a Legendrian submanifold $\Lambda_1$  and $\Phi^{-1}  (\Lambda_1)=\Lambda^{-\eps}\cup\Lambda^\eps$. Note that if we modify the embedding $\Phi$ as
 $\wt\Phi(x,y,z)=\Phi(x,y,z-\eps)$ we still get a contact form preserving embedding  $\wt\Phi:(Q_C,\alpha_\st)\to (Y,\alpha)$  for which    $\wt \Phi^{-1}  (\Lambda_1)=\Lambda^{-2\eps}\cup\Lambda^0$.
            \end{proof}
             \begin{proof}[Proof of Proposition \ref{prop:balancing} for $n>3$]
                Let $X_-$  be  a negative Liouville end of $X$ bounded by 
   a contact slice $Y\subset X$     such that $f$ is cylindrical below it.
    Denote $\Lambda:=f^{-1}(Y)$.    According to Lemma \ref{lm:small-action} for any $\eps$ there exists a Hamiltonian  regular homotopy of $f$ into a Lagrangian  immersion  with transverse  self-intersection points of action $<\eps$. Moreover, the number of self-intersection points grows proportionally to $\frac1{\eps^3}$ when $\eps\to0$. For a sufficiently small $C>0$ there exists a contact form preserving embedding $(Q_C,\alpha_\st)\to (Y\setminus\Lambda, \alpha:=\lambda|_Y)$.
     Note that given an integer $N>0$ and a positive  $\eps<\frac CN$ there exists contact form preserving embeddings of $N^n$ disjoint copies of $(Q_\eps,\alpha_\st)$ into $(Q_C,\alpha_\st)$, i.e. when   decreasing $\eps$ the number  of domains   $(Q_\eps,\alpha_\st)$  which can be packed into $ (Y\setminus\Lambda, \alpha)$ grows proportionally to $\eps^{-n}$, which is greater than $\eps^{-3}$ by assumption.    Hence for  a sufficiently small $\eps$ we can modify the Lagrangian immersion $f$, so that the action of all its self-intersection points are $<\eps$, and at least $\SI(f)$  disjoint Darboux neighborhoods isomorphic to $Q_{12\eps}$ which do not intersect   $\Lambda$ can be  packed into 
      $ (Y, \alpha)$. We will denote the number of self-intersection points by $N$ and the corresponding $Q_{12\eps}$-neighborhoods by $U_1,\dots, U_N$. Notice that for a sufficiently small $\theta>0$ there exists a Liouville form preserving embedding   $((0,1+\theta)\times Y,  t\alpha)\to (X,\lambda)$ which sends $Y\times 1$ onto $Y$.
          
 For each intersection point $p_i\in f(L)$, $i =1,\ldots N$, we will find a compactly supported Hamiltonian regular homotopy to balance each intersection point  $p_i$   without changing the action of  the other intersection points. Recall $0<a_\SI(p_1,f)<\eps$.  Using Lemma \ref{lm:parallel-slices} we isotope the Lagrangian cylinder $(0,1+\theta)\times \Lambda)$ via a Lagrangian  isotopy supported in a neighborhood of $Y\times 1$ so that:
 \begin{description}
 \item{-}   the deformed cylinder
 $\wt \Lambda$ intersects $Y$ transversely along a Legendrian submanifold $\wt\Lambda$;
\item{-} for a sufficiently small $\sigma>0$   and each $i=1,\dots, N$,
the cylinder  $\wt \Lambda$ intersects $U_i=Q_{12\eps}$ along  Legendrian planes
    $\Lambda^0= \{y=0,z=0\}$ and $\Lambda^{-\sigma}= \{z=-\sigma ,y=0\}$.
    \end{description}
   We can further deform the Lagrangian $\wt L$ to make it cylindrical in
   $[\frac12,1]\times Y$, and hence,  we get embeddings $  (\left[\frac12,1\right]\times Q_{12\eps}, t\alpha_\st)\to ((0,1]\times Y,t\alpha)$ such that 
the intersections $ (\left[\frac12,1\right]\times U_i, \alpha_\st) $  with $\wt L$   coincide with the Lagrangians $L^0$ and $L^{-\delta}$ from Lemma \ref{lm:loose-trick-model}. 
    
There are two cases, depending on the sign of the intersection; suppose first that the self-intersection index at the point $p_i$ is negative.
     Then we apply Lemma \ref{lm:loose-trick-model} with $k=0$
     and construct a  cylindrical at $-\infty$  and  fixed everywhere except  $L^{-\delta}$ and $\Lambda^{-\delta}\times\left(0,\frac12\right]$  Hamiltonian regular homotopy of  the immersion
     $f$ which   deforms  $ L^{-\delta}$ to 
     $\wt L^{-\delta}$  such   that $L^0$ and $\wt L^{-\delta}$ positively intersect at 1 point $B_0$ of action $a_\SI(B_0,f)=a_\SI(p_i,f)$. Hence, the point $B_0$ balances $p_i$. Notice that this homotopes $\Lambda$ to another Legendrian $\wt \Lambda$, and in fact $\wt \Lambda$ will never be Legendrian  isotopic to $\Lambda$ (after a   balancing  of a sigle intersection point; we show below that it will be isotopic after all intersection points are balaned).
     
     If the self-intersection  index of $p_i$ is positive we
     first  apply Lemma \ref{lm:2points} to create two new 
    intersection points 
     $p_+$ and $p_-$ of index $1$ and  $-1$ and action equal to $A-\sigma$ and 
     $A+\sigma$ respectively, for some $A \in (a_\SI(p_i, f), a_\SI(p_i, f)+4\epsilon)$ and sufficiently small $\sigma>0$. We then apply Lemma \ref{lm:loose-trick-model} with $k=2$ and create 3 new intersection points  $B_0, B_1, B_2$ of indices $1,-1,-1$ and of action $A+\sigma$, $A - \sigma$ and 
   $ a_\SI(p_i,f)$, respectively. Then $(p_i, B_2)$, $(p_+, B_1)$
   and $(p_-, B_0)$ are balanced Whitney pairs. 
   
   In the course of the above proof, $\Lambda$ is homotoped to the Legendrian $\tilde{\Lambda}$ at $-\infty$. In order to make the  constructed Hamiltonian   homotopy  of our Lagrangian  fixed at $-\infty$, it suffices to show that $\Lambda$ is Legendrian isotopic to $\tilde{\Lambda}$, because we can then apply Lemma \ref{lm:Leg-Lag} to undo this homotopy near $-\infty$. Assume that $\Lambda$ has a loose component and $I(f) = 0$. In the course of the above proof we only need to homotope a single component of $\Lambda$ of our choosing; we choose the component of $\Lambda$ which is loose. Obviously we can also fix a universal loose Legendrian embedded in this component of $\Lambda$, thus the corresponding component of $\tilde{\Lambda}$ is also loose. Using part (ii) of Proposition \ref{prop:Murphy}, it only remains to show that $\Lambda$ is formally Legendrian isotopic to $\tilde{\Lambda}$. Because the algebraic count of self intersections of $f$ is zero the homotopy from $\Lambda$ to $\tilde{\Lambda}$ also has an algebraic count of zero self-intersections. This implies that they are formally isotopic; see Proposition 2.6 in \cite{Murphy-loose}. \end{proof} 
   
   To deal with the case $n=3$ we will need an additional lemma.
   Let us denote by    $P(C)$   the polydisc $\{p_i^2+q_i^2\leq \frac{C}\pi,\; i=1,\dots, n\} \subset\R^{2n}_\st$.
    \begin{lemma}\label{lm:inserting-QC}
    Let $(X,\omega)$ be a symplectic  manifold with a negative Liouville end,  $Y\subset X$ a contact slice, and $\lambda$ is the corresponding Liouville form on a neighborhood  $\Omega\supset X_-$ in $X$.   Suppose that there exists a symplectic embedding $\Phi:P(C)\to  X_+\setminus Y$. Let  $\Gamma$ be an embedded path  in $X_+$ connecting  a point   $a\in Y$ with a point in $b\in  \p \wt P$, $\wt P:=\Phi(P(C))$.   Then for any  neighborhoods $U\supset(\Gamma\cup\wt P)  $ in $X_+$ 
    there exists a Weinstein  cobordism $(W,\om, \wt X,\phi)$ such that
     \begin{enumerate}
     \item  $W\subset X_+\cap (U\cup\Omega)$, $\p_-W=Y$;
     \item the Liouville form $\wt \lambda=\iota(\wt X)\omega$ coincides with  $\lambda$ near $Y$ and on $\Omega\setminus U$;
    \item $\phi$ has no critical points;   
    \item the  contact manifold  $(\wt Y:=\p_+W, \wt\alpha:=\wt\lambda|_{\wt Y})$ admits a contact form preserving embedding $(Q_{a},\alpha_\st)\to(\wt Y,\wt\alpha)$ for any $a<\frac C2$. 
    \end{enumerate}     \end{lemma}
    %%%%%%%%%%%%%%%
   
    \begin{proof} For any $b \in (a, \frac C2)$ the domain   $U_{b}:=
    \{|q_i|\leq 1, |p_i|< b;\; i=1,\dots, n\}\subset\R^{2n}_\st$ admits a symplectic embedding $H:U_b\to\Int P(C)$.
        Denote $\p_n U_b:=\{p_n=b\}\cap\p \overline U_b$. Consider a Liouville form
    $\mu=\sum\limits_1^n(1-\sigma) p_i dq_i-\sigma q_idp_i=\sum\limits_1^n   p_i dq_i-\sigma d\left(\sum\limits_1^n p_iq_i\right)$, where  a sufficiently small $\sigma>0$ will be chosen later.
    Then  
       $$\beta:=\mu|_{\p_nU_b} = d\left((b-\sigma) q_n-\sigma \sum\limits_1^{n-1} p_iq_i\right)+\sum\limits_1^{n-1}   p_i dq_i.$$
      
      Let us verify  that for a sufficiently small $\sigma>0$ there exists a contact form preserving embedding
       $(Q_a,\alpha_\st)\to (\p_nU_b,\beta)$. Consider the map
     $\Psi:Q_a\to  \R^{2n}_\st$ given by the formulas
     $$p_i=-y_i, q_i=x_i, \; i=1,\dots, n-1,p_n=b, q_n=\frac{z}{b-\sigma}-
     \frac{\sigma}{b-\sigma}\sum\limits_1^{n-1}x_iy_i.$$  
      Note that $|q_n|\leq \frac {a + a \sigma(n-1)}{b-\sigma}<1$ if $\sigma<\frac{b-a}n$.   Hence, if 
      $(x,y,z)\in Q_a$ we have $$|p_i|\leq a<b, |q_i|\leq 1\;\hbox{ for } i=1,\dots, n-1,  p_n=b, |q_n|<1,$$ i.e. $\Psi(Q_a)\subset \p_nU_b$. On the  other hand $$\Psi^*\mu=\Psi^*\beta =
       d\left( {z+\sigma\sum\limits_1^{n-1}x_iy_i} -\sigma \sum\limits_1^{n-1} x_iy_i\right)-\sum\limits_1^{n-1}   y_i dx_i=\alpha_\st.$$

     %%%%%%%%
    There exists a domain $\wh U_b$, diffeomorphic to a ball with smooth boundary, such that     \begin{itemize} 
    \item $U_b\subset\wh U_b\subset U_{b'} $ for some $b'\in(b,\frac C2)$; 
       \item $\p\wh U_b \supset\p_n U_b;$
    \item $\wh U_b$ is transverse to the Liouville field $T$, 
    $\omega$-dual to the Liouville form $\mu$.     \end{itemize}
   Note that there exists a Lyapunov function $\psi:\wh U_b\to\R$ for $T$ such that   $(\wh U_b,\om,T, \phi)$ is a Weinstein domain.
    
     Denote $\wt U_b:=\Phi (H(U_a))\Subset X_+$.
    We can assume that the path $\Gamma$ connects a point on $Y$ with a point on $\p\wt U_b\setminus \Phi(H((\p_n U_b))$.
 
 We modify the Liouville form $\lambda$, making it equal to $0$  on the path $\Gamma$ and equal to $\Phi_*H_*\mu$ on $\wt U_b$.  Next, we use Lemma \ref{lm:surgery} to construct the  required   cobordism $(W,\om, \wt X,\phi)$  by  connecting $X_-$  and $\wh U_b$ via a Weinstein surgery along $\Gamma$, and then apply Proposition  \ref{prop:cancellation} to cancel the zeroes of the Liouville field $\wt X$. As a result we ensure properties (i)--(iii). In fact, property (iv) also holds. Indeed, by construction $\p_+W\supset \Phi(H(\p_n U_b))$, and hence there exists a contact form preserving embedding $(Q_a,\alpha_\st)\to (\p_+W,\wt \alpha:=\iota(\wt X)\omega|_{\p_+W})$.  
      \end{proof}

 \begin{proof}[Proof of Proposition \ref{prop:balancing} for $n=3$]
    The problem in the case $n=3$ is that we cannot get sufficiently many disjoint contact neighborhoods  $Q_C$ embedded into $Y$ to  balance all the intersection points. Indeed, both the number of intersection  of action $<\eps$ and the number of $Q_{12\eps}$-neighborhoods one can pack into  contact slice $Y$ grow as $\eps^{-3}$ when $\eps\to 0$.    However, using the inifinite Gromov width assumption we can cite Lemma \ref{lm:inserting-QC} to modify $Y$ so that it  would contain a sufficient number of  disjoint neighborhoods isomorphic to $Q_{12\eps}$.  Indeed, suppose that there are $N$ double points af action $<\eps$. By the infinite Gromov  width assumption  there exists $N$ disjoint embeddings  of polydiscs $P( 24\eps)$  into $X_+\setminus f(L)$.
  
    Using Lemma  \ref{lm:inserting-QC}, we modify the Liouville form $\lambda$ into $  
  \wt\lambda$ away from $f(L)$, so that $(X,\wt\lambda)$ admits a negative end bounded by a contact     slice $\wt Y$ such that  there exists $N$ disjoint embeddings 
  $(Q_{12\eps},\alpha_\st)\to(\wt Y,\wt\alpha)$ preserving the contact form. The rest of the proof is identical to the case $n>3$.
   \end{proof}

\section{Proof of  main theorems}\label{sec:proofs}
 \begin{proof}[Proof of Theorem  \ref{thm:main-imm}]

 We first use  Proposition  \ref{prop:balancing} to make the Lagrangian immersion $f$ balanced and then use the following
    modified Whitney trick to eliminate
each balanced   Whitney pair.

Let $p,q\in X$ be a balanced Whitney  pair, $p^0,p^1\in L$ and $q^0,q^1\in L$ the  pre-images of the self-intersection points $p,q$, and $\gamma^0,\gamma^1:[0,1]\to L$ are the corresponding paths such that 
  $ \gamma^j(0)=p^j  , \gamma^j(1)=q^j$ for    $j=0,1$,   
 the intersection index of $df(T_{p^0}L)$ and  $df(T_{p^1}L)$ is equal to $1$  and  the intersection index of $df(T_{q^0}L)$ and  $df(T_{q^1}L)$ is equal to $-1$.    
  Recall that according to our convention  we are  always ordering the pre-images of double points in such a way that their action is positive.
  
Choose a contact slice $Y$, and consider a path $\eta:[0,1]\to L$ connecting a point in the loose component $\Lambda$ of $\p L_+$ with $p^0$ such that $\overline{\eta}:=f\circ\eta$ coincides with a trajectory of $Z$ near   the point  $\overline{\eta}(0)$, and then modify the   Liouville form $\lambda $,  keeping it fixed on $X_-$,  to make it equal to $0$ on $\overline{\eta}$.  We further modify $\lambda$ in a neighborhood of $\og^0$ making it $0$ on $\og^0$, where we use the notation
  $\og^0:=f\circ \gamma^0$, $\og^1:=f\circ\gamma^1$.
Note that this is possible because $Y \cup \overline{\eta} \cup \og^0$ deformation retracts to $Y$.   Assuming that this is done, we  observe that
  $ \int\limits_{\og^1} \lambda=  \int\limits_{\og^0} \lambda=0$.
      
 Next, we use Lemma \ref{lm:surgery-index0} to construct Darboux charts $B_p$ and $B_q$ centered  at the  points $p$ and $q$ such that the the intersecting branches in these coordinates look like   coordinate Lagrangian planes $\{q=0\}$ and $\{p=0\}$ in the standard $\R^{2n}$. Set  $\lambda_\st:=\frac 12  \sum\limits_1^n p_i dq_i- q_i dp_i$. Then the corresponding to it  Liouville vector field $Z_\st=\frac12\sum\limits_1^nq_i\frac{\p}{\p q_i}+p_i\frac{\p}{\p p_i} $   is tangent to the Lagrangian planes through the origin.
 
We have $\lambda_\st-\lambda=dH$ in  $B_p\cup B_q$. Choosing a cut-off function $\alpha$ on $B_p\cup B_q$ which is equal to $1$ near $p$ and $q$ and equal to $0$ near $\p B_p\cup\p B_q$ we define $\lambda_1:=\lambda+d(\alpha H)$. The Liouville structure $\lambda_1$ coincides with the standard structure $\lambda_\st$ in smaller balls around  the points  $p$ and $q$, and with $\lambda$ near $\p B_p\cup\p B_q$. 
 
Next, we use Lemma \ref{lm:surgery} to modify the Liouville structure $\lambda_1$ in neighborhoods of paths  $\og^0$ and $\og^1$ and create Weinstein domain $C$ by attaching handles of index $1$ with $\og^0$ and $\og^1$ as their cores.  The corresponding Lyapunov function on $C$ has  two critical points of index $0$, at $p$ and $q$, and two critical points of index $1$, at the centers of paths $\og^0$ and $\og^1$.   
Note that the property $\int\limits_{\og^j}\lambda_1=0$, $j=0,1$, is crucial in order to apply Lemma \ref{lm:surgery}.

Next, we  choose an embedded 
  isotropic disc $\Delta\subset X_+\setminus \Int C$ with boundary in $\p C$,   tangent to $Z$ along the boundary $\p \Delta$, and such that $\p\Delta$ is isotropic, and   homotopic in $C$ to the loop $\og^0 \cup \og^1$. We then  again use Lemma \ref{lm:surgery} to attach to $C$ a handle of index $2$ with the  core $\Delta$. The resulted Liouville domain $\wt C$ is diffeomorphic to the $2n$-ball. Moreover, according to Proposition \ref{prop:cancellation}  the Weinstein structure on $\wt C$ is homotopic to the standard one via a homotopy fixed on $\p\wt C$.
In particular,  the contact structure induced on the sphere  $\p \wt C$ is
the    standard one. The immersed Lagrangian manifold $f(L)$ intersects $\p\wt C$ along two Legendrian spheres $\Lambda^0$ and $\Lambda^1$, each of which is the standard Legendrian unknot which bounds an   embedded Lagrangian disc  inside $\wt C$.  These two discs intersect  at two points, $p$ and $q$. Note that the Whitney trick allows us to disjoint these discs by a smooth (non-Lagrangian) isotopy fixed on their boundaries. In particular, the spheres  $\Lambda^0$ and $\Lambda^1$ are smoothly unlinked. If they were unlinked as Legendrians we would be done. Indeed, the Legendrian unlink in $S^{2n-1}_\std$ bounds two disjoint exact Lagrangian disks in $B^{2n}_\std$. Unfortunately (or fortunately, because this would kill Symplectic Topology as a subject!), one can show that it is impossible to unlink $\Lambda^0$ and $\Lambda^1$ via a Legendrian isotopy. 
  
The path $\oeta $ intersects $\p\wt C$ at a point in $\Lambda^0$. Slightly abusing the notation we will continue using  the notation $\oeta$ for  the part of $\oeta$ outside the ball $\wt C$. We then use Lemma \ref{lm:surgery} one more time to modify $\lambda_1$ by attaching a handle of index $1$ to $X_-\cup \wt C$ along $\oeta$. As a result, we create inside $X_+$ a Weinstein cobordism $W$ which contains $\wt C$, so that $\p_-W=Y$  and $\wt Y:=\p_+W$ intersects $f(L)$ along a $2$-component  Legendrian link. 
%%%%
One of its components is $\Lambda^1$, and the other one is the connected sum of the loose Legendrian $\Lambda$ and the Legendrian sphere $\Lambda^0$, which we denote by $\wt \Lambda$.
  Again applying   Proposition \ref{prop:cancellation} we can deform the Weinstein structure on $W$ keeping it fixed on $\p W$ to kill both critical points  inside $W$. Hence all trajectories of the (new) Liouville vector field $Z$ inside $W$ begin at $Y$ and end at $\wt Y$, and thus $W$ is Liouville isomorphic to $\wt Y \times [0, T]$ for some $T$ (with Liouville form $e^t\lambda_1$, $t \in [0, T]$). We also note that the intersection of $f(L)$ with $W$ consists of two embedded Lagrangian submanifolds $A$ and $B$ transversely intersecting in the points $p,q$, where
  \begin{itemize}
  \item $A$ is diffeomorphic to the cylinder $\Lambda \times [0,1]$, $A\cap Y=\Lambda$ and $A\cap\wt Y=\wt\Lambda$;
  \item $B$ is a disc bounded by the Legendrian sphere  
   $\Lambda^1=B\cap \wt Y$.
  \end{itemize}
 
 The Legendrian $\wt \Lambda$ is smoothly unlinked with $\Lambda^1$. Since $\wt \Lambda$ is loose, Proposition \ref{prop:Murphy} implies that there is a Legendrian isotopy of $\wt \Lambda$ to $\wh \Lambda$ which is disjoint from a Darboux ball containing $\Lambda^1$. We realize this isotopy by a Lagrangian cobordism $A_1$ from $\wt \Lambda$ to $\wh \Lambda$ using Lemma \ref{lm:Leg-Lag}, and also realize the inverse isotopy by a Lagrangian cobordism $A_2$ from $\wh \Lambda$ to $\wt \Lambda$. For some $\wt T$, these cobordisms embed into $\wt Y \times [0, \wt T]$. Inside $\wt Y \times [0, 2 \wt T + 2T]$, we define a cobordism $\wt A$ from $\Lambda$ to $\wt \Lambda$, built from the following pieces. 
 \begin{itemize}
 \item $\wt A \cap \wt Y \times [0, T] = A$, 
 \item $\wt A \cap \wt Y \times [T, \wt T+T] = A_1$,
 \item $\wt A \cap \wt Y \times [\wt T + T, \wt T +2T] = \wh \Lambda \times [\wt T + T, \wt T +2T]$,
 \item  $\wt A \cap \wt Y \times [\wt T + 2T, 2\wt T +2T] = A_2$.
 \end{itemize}
 We then define $\wt B$ by 
 \begin{itemize}
 \item $\wt B \cap \wt Y \times [0, \wt T + T] = \varnothing$,
 \item $\wt B \cap \wt Y \times [\wt T + T, \wt T + 2T] = B$,
 \item $\wt B \cap \wt Y \times [\wt T + 2T, 2\wt T + 2T] = \Lambda^1 \times [\wt T + 2T, 2\wt T + 2T]$.
 \end{itemize}
 
       \begin{figure}
 \center{\includegraphics[height=85mm]{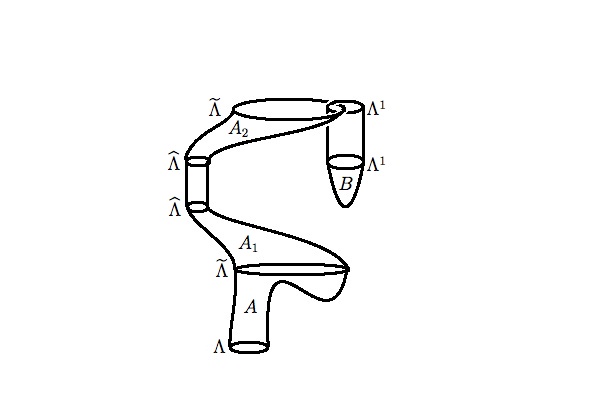}} \caption{The cobordisms $\wt A$ and $\wt B$.} \label{fig:cobord}
\end{figure}
 
 A schematic of these cobordisms is given in Figure ~\ref{fig:cobord}. After elongating $W$ (which can be achieved by choosing a contact slice closer to $-\infty$), $\wt A \cup \wt B$  can be deformed  to $A \cup B$ via a Hamiltonian  compactly supported regular homotopy fixed on the boundary. We then define $\wt f:L \to X$ to be equal to $f$ everywhere, except the portions of $L$ which are mapped to $A$ and $B$ are instead mapped to $\wt A$ and $\wt B$, respectively.
 \end{proof}
  
 \medskip
 
\begin{proof}[Proof of  Theorem \ref{thm:main}]
We first use Gromov's $h$-principle for Lagrangian immersions \cite{Gr-PDR} to find a compactly supported regular homotopy   
starting at $f$ and ending at a Lagrangian immersion $\wt f $ with the prescribed action class $A(f)$ (or the action class $a(f)$ in the Liouville case). More precisely, let us choose a triangulation of $L$. There are finitely many simplices of the triangulation which cover the compact part of $L$ where the embedding $f$ is not yet Lagrangian. Let $K$ be the polyhedron which is formed by these simplices. Using the $h$-principle for open Lagrangian immersions, we first isotope $f$ to an  embedding which is Lagrangian near the $(n-1)$-skeleton of $K$, realizing the given (relative) action class. Let us inscribe an $n$-disc  $D_i$ in each of  the $n$-simplices of $K$, such that the embedding $f$ is already Lagrangian near $\p D_i$. Next, we thicken $D_i$ to {\it disjoint} $2n$-balls $B_i\subset X$ intersecting $f(L)$ along $D_i$. We then apply Gromov's $h$-principle for Lagrangian immersions in a relative form to find for each $i$ a fixed near the boundary regular homotopy $D_i \to B_i$ of $D_i$ into  a Lagrangian immersion.  Note that all the self-intersection points of the resulted Lagrangian immersion $\wt f$  are localized inside the ball $B_i$ and images of different discs $D_i$ and $D_j$ do not intersect.

  Let us choose a negative  end $X_-$, bounded by a contact slice $Y$  in such a way that the immersion  $\wt f$  is cylindrical  in it and $X_-\cap\bigcup B_i=\varnothing$. Denote  $L_-:=\wt f^{-1}(X_-), \Lambda_-=\p L_-$.
   Let us choose a universal loose Legendrian $U \subset  Y$ for the Legendrian submanifold $\Lambda_-\subseteq Y$. Denote $\wt\Lambda_-=\Lambda_-\cap U$.
   Let  $V_-:=\bigcup\limits_0^{\infty} Z^{-s}(U)\subset
    X_-$ be the domain in $X_-$ formed by all negative trajectories of $Z$ intersecting $U$.   Let us choose disjoint paths $\Gamma_i$ in $L\setminus \Int (L_- \cup\bigcup_i D_i)$ connecting  some points in $\wt\Lambda_-$ with points $z_i \in \p D_i$ for each $n$-simplex in $K$. Choose small tubular neighborhoods $U_i$ of $\wt f(\Gamma_i)$ in $X$
    
  Set  $$\wt X:= V_-\cup  \bigcup_i (B_i\cup U_i)\;\;\hbox{and}
 \;\;\wt L:= \wt f^{-1}(\wt X).$$  
 The manifold $\wt X$ deformationaly retracts to $V_-$ and hence $\wt X$ is contractible and the Liouville form $\lambda|_{V_-}$ extends as a Liouville form for $\omega$ on the whole manifold $\wt X$. We will keep the notation $\lambda$ for the extended form. Thus $\wt L$ is an exact Lagrangian immersion into the contractible Liouville manifold $\wt X$, cylindrical at $-\infty$ over a loose Legendrian submanifold of $U$.
 Moreover, $L$ is diffeomorphic to $\R^n$, and outside a compact set the immersion is equivalent to the standard inclusion $\R^n\hookrightarrow\R^{2n}$. We also note that $I(\wt f|_{\wt L}: \wt L\to\wt X) =  0$ since this immersion is regularly homotopic to the smooth embedding $f|_{\wt L}:\wt L\to\wt X$.
 
 Applying Theorem  \ref{thm:main-imm} to $\wt f|_{\wt L}$ 
we find an exact Lagrangian embedding $\wh f$ which is regularly Hamiltonian homotopic to $\wt f|_{\wt L}$  via a regular homotopy   compactly supported in $\wt X$. We further note
that the embeddings $\wh f$ and $f:\wt L\to\wt X$ are isotopic relative the boundary.
Indeed, it follows from the $h$-cobordism theorem that an embedding
$\R^n\to \R^{2n}$ which coincides with the inclusion outside a compact set  and which is regularly homotopic to it via a compactly supported homotopy is isotopic to the inclusion relative infinity. 

Slightly abusing notation we define $\wh f:L \to X$ to be equal to $\wt f$ on $L\setminus \wh L$. This Lagrangian embedding is isotopic to $f$ via an isotopy fixed outside a compact set. Finally we note that $d \wt f:TL \to TX$ is homotopic to $\Phi_1$ since it is constructed with the $h$-principle for Lagrangian immersions, and $d \wh f$ is homotopic to $d \wt f$ since they are regularly Lagrangian homotopic.
 \end{proof}
 
 Next, we deduce Theorem  \ref{thm:caps} from 
    Theorem \ref{thm:main}.
 
 \begin{proof}[Proof of Theorem \ref{thm:caps}]
  Let $B$ be the unit  ball in $\R^{2n}$. The triviality of the bundle $T(L)\otimes\C$ is equivalent to existence of a Lagrangian homomorphism $\Phi:TL
 \to T\C^n$. We can assume that $\Phi$ covers a
  map $\phi:L\to\C^n\setminus \Int B$ such that $\phi(\p L)\subset \p B$.
   Let $v\in TL|_{\p L}$ be the inward normal vector field to $\p L$ in $L$, and $\nu$ an outward normal to the boundary $\p B$ of the ball $B\subset\C^n$.
    Homomorphism $\Phi$ is homotopic to a Lagrangian homomorphism, which will still be denoted by $\Phi$, sending $v$ to $\nu$. Indeed,  the  obstructions to that lie   in trivial  homotopy group $\pi_j(S^{2n-1})$,  $j\leq n-1$. Then $\Phi|_{T\p L}$ is a Legendrian homomorphism $T\p L\to\xi$, where $\xi$ is the  standard contact structure on the sphere  $\p B$ formed by its  complex tangencies.
    Using   Gromov's $h$-principle for Legendrian embeddings we can, therefore, assume that $\phi|_{\p L}:\p L\to \p B$ is a Legendrian embedding, and then, using Gromov's $h$-principle for Lagrangian immersions deform $\phi$ to an exact Lagrangian immersion  $\phi:L\to\C^n\setminus \Int B$ with Legendrian boundary in $\p B$ and tangent to $\nu$ along the boundary. Finally, we use Theorem \ref{thm:main} to make $\phi$ a Legrangian embedding.
 \end{proof}

 \section{Applications}\label{sec:applications}
 
\subsubsection*{Lagrangian embeddings with a conical singular point}

Given a symplectic manifold $(X,\om)$ we say that $L\subset  M$ is a {\it  Lagrangian submanifold with  an isolated conical point} if  it is  a Lagrangian submanifold away from a point $p\in L$, and there exists a symplectic embedding $f:B_\eps\to X$ such that $f(0)=p$  and $f^{-1}(L)\subset B_\eps$ 
  is a Lagrangian cone. Here $B_\eps$ is the ball of radius $\eps$ in the standard symplectic $\R^{2n}$.
  Note that this cone  is automatically a cone over a Legendrian sphere in  the sphere $\p B_\eps$ endowed with the  standard contact structure given by the restriction to $\p B_\eps$  of the Liouville form
  $\lambda_\st=\frac12\sum\limits_1^n(p_idq_i-q_idp_i)$.
 
  As a special case of Theorem \ref{thm:caps} (when $\p L$ is a sphere) we get  
 
\begin{cor}\label{cor:conic}
Let $L$ be an $n$-dimensional, $n>2$, closed  manifold such that the complexified tangent bundle $T^*(L\setminus p)\otimes\C$ is trivial. Then $L$ admits an exact Lagrangian embedding into $\R^{2n}$ with exactly  one  conical  point. In particularly  a sphere admits a Lagrangian embedding to $\R^{2n}$ with one conical point for each $n>2$.
\end{cor}

 \subsubsection*{Flexible Weinstein cobordisms}
       
 The following notion of a flexible Weinstein cobordism is introduced in \cite{CieEli-Stein}.     
    
  A Weinstein cobordism  $(W,\om,Z,\phi)$ is called  {\it elementary} if there are
  no $Z$-trajectories connecting critical points.
   In this case stable manifolds of critical points intersect $\p_-W$ 
   along  isotropic in the contact sense submanifolds.  
   For each  critical point $p$ we call the intersection $S_p$ of its stable manifold with $\p_-W$  the  {\it attaching  sphere}.
    The attaching spheres for index $n$ critical points are Legendrian.

 An elementary  Weinstein cobordism  $(W,\om,Z,\phi)$ is called {\it flexible  } if  the attaching spheres for all index $n$ critical points in $W$ form a  loose Legendrian link in $\p_-W$.

  A Weinstein cobordism $(W,\om,Z,\phi)$ is called {\it flexible}
  if it can be partitioned into elementary Weinstein cobordisms:  $W=W_1\cup\dots\cup W_N$, $W_j:=\{c_{j-1}\leq \phi\leq c_j
\}, j=1,\dots, N$, $m=c_0<c_1<\dots< c_N=M$. Any subcritical Weinstein cobordism is by definition flexible.
  
  \begin{thm}\label{thm:self-embed}
  Let $(W,\om,Z,\phi)$ be a flexible Weinstein domain. Let $\lambda$ be the Liouville  form $\om$-dual to $Z$, and $\Lambda$ any other Liouville form such that the symplectic structures $\om$ and $\Omega:=d\Lambda$ are homotopic as non-degenerate (not necessarily closed) $2$-forms.
  Then  there exists an isotopy $h_t:W\to W$ such that $h_0=\Id$ and
  $h_1^*\Lambda=\eps\lambda+dH$ for a sufficiently small $\eps>0$ and a  smooth function $H:W\to\R$. In particular, $h_1$ is a symplectic embedding $(W,\eps\om)\to(W,\Omega)$.
  \end{thm}
  Recall that a  Weinstein cobordism $(W,\om,Z,\phi)$ is called a {\it Weinstein domain} if $\p_-W=\varnothing$.
\begin{cor}\label{cor:flexible-embed}
Let $(W,\om,Z,\phi)$ be a flexible Weinstein domain, and $(X,\Omega)$ any symplectic manifold of the same dimension. If this dimension is $3$ we further assume that $X$ has infinite Gromov width. Then any smooth embedding $f_0:W\to X$, such that  the form $f_0^*\Omega$ is exact and the differential $df:TW\to TX$ is homotopic to a symplectic homomorphism, is isotopic to a symplectic embedding $f_1:(W,\eps\om)\to (X,\Omega)$ for a sufficiently small $\eps>0$.
Moreover, if $\Omega=d\Theta$  then the embedding $f_1$ can be chosen in such a way that  the 1-form $f_1^*\Theta-i(Z)\om$ is exact. If, moreover,
the  $\Omega$-dual to $\Theta$ Liouville vector field   is complete then
the embedding $f_1$ exists for an arbitrarily large constant $\eps$.
\end{cor}

    \begin{proof}[Proof of Theorem \ref{thm:self-embed}] 
   
   Let us decompose $W$ into flexible elementary cobordisms:
   $W=W_1\cup\dots\cup W_k$, where $W_j=\{c_{j-1}\leq \phi\leq c_j\}$, $j=1,\dots, k$ for a sequence of regular values $c_0<\min\phi<c_1<\dots<c_k=\max\,\phi$ of the function $\phi$. Set $V_j=\bigcup\limits_1^j W_i$  for $j\geq 1$ and $V_0=\varnothing$.
     
     We will construct an isotopy $h_t:W\to W$ beginning from $h_0=\Id$
    inductively over cobordisms $W_j$, $j=1,\dots, k$. It will be convenient to parameterize the required isotopy by the interval $[0,2k]$.
 Suppose that  for some $j=1,\dots, k$ we already constructed an isotopy $h_t:W\to W$, $t\in[0,j-1]$   such that   $h^*_{j-1}\Lambda=\eps_{j-1}\lambda+dH$ on $V_{j-1}$. Our goal is  to extend it $[j-1,j]$ to ensure  that $h_j$ satisfies this condition on $V_j$. Without loss of generality we can assume that there exists only 1 critical point  $p$ of $\phi$ in $W_j$. Let $\Delta$ be the stable disc of $p$ in $W_j$ and $S :=\p\Delta \subset \p_-W_j$ the corresponding attaching sphere.   By assumption, $S $ is subcritical or loose. The homotopical condition implies that there is a family of injective homomorphisms $\Phi_t:T\Delta\to TW$, $t\in[j-1,j]$,  such that $\Phi_{j-1}=dh_{j-1}|_{\Delta_j}$, and $\Phi_j:T\Delta_j\to (TW,\Omega)$ is an isotropic homomorphism. We also note that  the cohomological condition implies that $\int\limits_\Delta\Omega=0$ when $\dim\Delta=2$. Then, using Theorem \ref{thm:main} when $\dim\Delta=n$    and Gromov's $h$-principle, \cite{Gr-PDR}, for isotropic embeddings in the subcritical case, we can construct an isotopy $g_t:\Delta  \to W_j$, $t\in[j-1,j]$,  fixed  at $\p \Delta $, such that $g_{j-1}=h_{j-1}|_{\Delta}$ is the inclusion and    the embedding
 $g_j:\Delta\to (W,\Omega)$ is isotropic.  Furthermore, there exists    a neighborhood  $U\supset\Delta $ in $W_j$ such that
 the isotopy $g_t$ extends  as a  fixed on $W_{j-1}$ 
  isotopy $G_t:W_{j-1}\cup U\to W$ such that $G_t|_{\Delta}=g_t$, $G_t|_{W_j}=h_{j-1}|_{W_{j-1}}$ for $t\in[j-1,j]$, $G_{j-1}|_U=h_{j-1}|_U$ and     $h_j:(W_{j-1}\cup U ,\eps_{j-1}\om)\to (W, \Omega)$ is a symplectic embedding. Choose  a sufficiently large $T>0$ we have $Z^{-T}(W_j)\subset W_{j-1}\bigcup U_j$,  and hence $h_j\circ  e^{-T}|_{V_j}$ is    a  symplectic embedding $(W_j,  \eps_j\om)\to (W,\Omega)$, where we set
  $\eps_j:=e^{-T}\eps_{j-1}$. Then  we can define the required isotopy $h_t:W\to W$, $t\in[j-1,j]$, which satisfy the property that $h_j|_{V_j}$ is a symplectic embedding  $(V_j, \eps_j\om)\to (W,\om)$   by setting $$h_t=
 \begin{cases} h_{j-1}\circ Z^{-2T(t-j+1)} & \hbox{ for}\;  t\in[j-1,j-\frac12],\cr
 G_t\circ Z^{-T}& \hbox{for}  \; t\in[j-\frac12,j].
 \end{cases}
 $$
  %%%
 
\end{proof}

%%%

\end{document}